\documentclass[12pt]{amsart}

\usepackage{aliascnt,amsmath,amssymb,amsthm,amsfonts,enumerate,mathrsfs,latexsym,mathtools,bm,xcolor,shuffle}
\usepackage{stmaryrd}	
\usepackage[abbrev]{amsrefs}
\usepackage{appendix}

\usepackage[OT2,T1]{fontenc}
\usepackage{scalerel}

\usepackage[marginparwidth=0pt,margin=24truemm]{geometry}

\definecolor{mylinkcolor}{RGB}{16, 156, 81}
\definecolor{mycitecolor}{RGB}{20, 80, 140}

\usepackage{hyperref}
\usepackage[nameinlink]{cleveref}
\hypersetup{
    setpagesize=false, 
    bookmarksnumbered=true, 
    bookmarksopen=true, 
    colorlinks=true, 
    linkcolor=mylinkcolor, 
    citecolor=mycitecolor,}
\usepackage{autonum}
\numberwithin{equation}{section}

\usepackage{tikz}
\usetikzlibrary{intersections,calc,arrows.meta,decorations,decorations.pathreplacing}

\theoremstyle{plain}

\newtheorem{thm}{Theorem}[section]
\crefname{thm}{Theorem}{Theorems}

\newaliascnt{lem}{thm} 
\newtheorem{lem}[lem]{Lemma}
\aliascntresetthe{lem}
\crefname{lem}{Lemma}{Lemmas}

\newaliascnt{prop}{thm}
\newtheorem{prop}[prop]{Proposition}
\aliascntresetthe{prop}
\crefname{prop}{Proposition}{Propositions}

\newaliascnt{cor}{thm}
\newtheorem{cor}[cor]{Corollary}
\aliascntresetthe{cor}
\crefname{cor}{Corollary}{Corollaries}

\newaliascnt{fact}{thm}

\aliascntresetthe{fact}
\crefname{fact}{Fact}{Facts}

\newaliascnt{conj}{thm}
\newtheorem{conj}[conj]{Conjecture}
\aliascntresetthe{conj}
\crefname{conj}{Conjecture}{Conjectures}

\newaliascnt{dfn}{thm}
\newtheorem{dfn}[dfn]{Definition}
\aliascntresetthe{dfn}
\crefname{dfn}{Definition}{Definitions}

\newaliascnt{rem}{thm}
\newtheorem{rem}[rem]{Remark}
\aliascntresetthe{rem}
\crefname{rem}{Remark}{Remarks}

\newaliascnt{ex}{thm}
\newtheorem{ex}[ex]{Example}
\aliascntresetthe{ex}
\crefname{ex}{Example}{Examples}

\newtheorem{mainthm}{Main Theorem}
 
\newaliascnt{mainconj}{mainthm}   
 
\aliascntresetthe{mainconj}        
\crefname{mainconj}{Main Conjecture}{Main Conjectures}

\allowdisplaybreaks[2]

\everymath{\displaystyle}

\newcommand{\relmid}{\mathrel{}\middle|\mathrel{}}
\newcommand{\summ}[1]{\sum_{\substack{#1}}}

\newcommand{\odd}{\mathrm{odd}}
\newcommand{\even}{\mathrm{even}}
\newcommand{\otherwise}{\text{otherwise}}

\newcommand{\QQ}{\mathbb{Q}}
\newcommand{\ZZ}{\mathbb{Z}}
\newcommand{\RR}{\mathbb{R}}
\newcommand{\NN}{\mathbb{N}}
\newcommand{\CC}{\mathbb{C}}
\newcommand{\HH}{\mathbb{H}}

\newcommand{\sltwo}{\mathfrak{sl}_2}

\DeclareMathOperator{\Span}{Span}
\DeclareMathOperator{\Hom}{Hom}
\DeclareMathOperator{\SL}{SL}

\DeclareMathOperator{\id}{id}
\DeclareMathOperator{\Der}{Der}

\newcommand{\halm}{\ha^{\ge2,\mathrm{alm}}}
\newcommand{\der}{\theta} 

\newcommand\quotient[2]{
	\mathchoice
	{
		\text{\raise1ex\hbox{$#1$}\Big/\lower1ex\hbox{$#2$}}%
	}
	{
		#1\,/\,#2
	}
	{
		#1\,/\,#2
	}
	{
		#1\,/\,#2
	}
}

\DeclareMathOperator*{\bigsha}{\scalerel*{\shuffle}{\sum}}

\newcommand{\mes}{\mathcal{E}}
\newcommand{\bmes}{\mathcal{G}}

\newcommand{\MEI}{\mathfrak{g}} 

\newcommand{\mz}{\mathcal{Z}}
\newcommand{\Ds}{\operatorname{ds}}
\newcommand{\dep}{\operatorname{dep}}

\newcommand{\ha}{\mathfrak{H}}
\newcommand{\hha}{\widehat{\mathfrak{H}}}
\newcommand{\bMZV}{\zeta}
\newcommand{\qmzv}{\mathsf{qMZV}}
\newcommand{\cmes}{G^{q}}

\newcommand{\nega}{\mathsf{neg}}
\newcommand{\anti}{\mathsf{anti}}
\newcommand{\pari}{\mathsf{pari}}
\newcommand{\swap}{\mathsf{swap}}

\newcommand{\BIMU}{\mathsf{BIMU}}
\newcommand{\gaxit}{\mathsf{gaxit}}

\newcommand{\gilat}{\mathsf{gilat}}
\newcommand{\gila}{\mathsf{gila}}

\newcommand{\co}{\mathsf{co}}

\renewcommand{\Im}{\operatorname{Im}}

\newcommand{\A}{\mathcal{A}}
\DeclareRobustCommand{\ai}{\genfrac{[}{]}{0pt}{}}
\newcommand{\QA}{\mathbb{Q}\langle\mathcal{A}\rangle}
\newcommand{\fA}{\mathfrak{A}}

\newcommand{\I}{\mathcal{I}}
\newcommand{\bMTF}[1]{\mathcal{L}_{#1}(\tau)}
\newcommand{\bMEI}{\mathfrak{g}}
\newcommand{\bMES}{G}

\title{Relations and Derivatives of Multiple Eisenstein Series}

\author{Henrik Bachmann}
\address{Graduate School of Mathematics, Nagoya University, Nagoya, Japan.}
\email{henrik.bachmann@math.nagoya-u.ac.jp}

\author{Hayato Kanno}
\address{Mathematical Institute, Tohoku University, Sendai, Japan.}
\email{hayato.kanno.q1@dc.tohoku.ac.jp}

\author{Takumi Maesaka}
\address{Faculty of Mathematics, Kyushu University, Fukuoka, Japan}
\email{nozaki.takumi.912@s.kyushu-u.ac.jp}

\date{\today}


\begin{document}

\begin{abstract}
In this paper, we study multiple Eisenstein series, which build a natural bridge between the theory of multiple zeta values and modular forms. We prove a large family of relations among these series and give an explicit formula for their derivatives. This formula is expressed using the double shuffle structure and the Drop1 operator introduced by Hirose, Maesaka, Seki, and Watanabe. In particular, the space of multiple Eisenstein series is closed under the derivative.  Further we construct bi-multiple Eisenstein series, which give a realization of the formal multiple Eisenstein series as holomorphic functions on the upper half-plane, and we prove a conjecture of Okounkov on derivatives of $q$-analogues of multiple zeta values. Based on the derivative formula, we propose a family of linear relations that is conjectured to generate all linear relations among multiple Eisenstein series. Motivated by this conjecture, we introduce a space of formal multiple Eisenstein series and show that it is an $\mathfrak{sl}_2$-algebra.
\end{abstract}

\maketitle

\section{Introduction}

The purpose of this paper is to study relations and derivatives of multiple Eisenstein series, which were introduced by Gangl, Kaneko and Zagier in \cite{GKZ}. These are holomorphic functions on the complex upper half-plane that can be viewed as hybrids of multiple zeta values and classical Eisenstein series. For multiple zeta values, there exist various families of relations, such as the extended double shuffle relations (\cite{IKZ}), associator relations (\cite{Dr},\cite{Fu}), confluence relations (\cite{HS}), and Kawashima relations (\cite{Kaw}), each of which is conjectured to generate all relations among multiple zeta values. However, for multiple Eisenstein series, such a family has not been proposed before. In this paper, we extend some results from \cite{BT} and prove a large family of relations among multiple Eisenstein series (\Cref{cor:R}) by showing that their harmonic and shuffle regularizations are the same on a certain domain (\Cref{thm:comparison of regularization}). Furthermore, following \cite{BIM}, one expects that there exists an $\mathfrak{sl}_2$-action on the space of multiple Eisenstein series, and in particular, that the space is closed under the derivative. To this end, we construct a realization of the formal multiple Eisenstein series introduced in \cite{BIM} as holomorphic functions on the upper half-plane, called bi-multiple Eisenstein series (\Cref{thm:biMES}), and we prove an explicit formula for the derivative (\Cref{thm:der}), which shows in particular that the space of multiple Eisenstein series is indeed closed under the derivative. As an application, we prove a conjecture of Okounkov on the derivatives of his $q$-analogues of multiple zeta values (\Cref{cor:okounkov}). Moreover, we propose, based on numerical computations, that all relations among multiple Eisenstein series are obtained from the relations in \Cref{cor:R} together with their derivatives (\Cref{conj:all relation}). Motivated by this conjecture, we introduce a new variant of the algebra of formal multiple Eisenstein series and show that it is an $\mathfrak{sl}_2$-algebra (\Cref{thm:main2}). Building on these results, we conclude by defining a new regularization of multiple Eisenstein series, which we call multiple Eisenstein-diamond series $G^\diamondsuit$.

\subsection{Multiple zeta values}

For an integer $r\ge1$ and a tuple $(k_1,\dots,k_r)\in(\ZZ_{>0})^r$ satisfying $k_1\ge2$ (called an \emph{admissible index}), the \emph{multiple zeta value} is defined by
\begin{align}
    \zeta(k_1,\dots,k_r)\coloneqq\sum_{n_1>\cdots>n_r>0}\frac{1}{n_1^{k_1}\cdots n_r^{k_r}}.
\end{align}
These values are subject to many relations; for example, Euler showed that $\zeta(2,1)=\zeta(3)$. Denote the $\QQ$-algebra of all multiple zeta values by $\mz$. Conjecturally, all algebraic relations among multiple zeta values are a consequence of the \emph{extended double shuffle relations} (\cite{IKZ}). These relations are obtained (after possible regularization) from the two ways of expressing the product of multiple zeta values: the harmonic product from the iterated sum representation, and the shuffle product from the iterated integral representation of multiple zeta values.
First, we recall the usual algebraic setup for multiple zeta values, following \cite{IKZ}. Let $\ha^1\coloneqq\QQ\langle z_k\mid k\ge1\rangle$ be the non-commutative polynomial ring over $\QQ$ generated by the letters $\{z_k\mid k\ge1\}$. The harmonic product on $\ha^1$ is defined by $w\ast1=1\ast w=w$ for any word $w\in\ha^1$ and $z_ku\ast z_lv=z_k(u\ast z_lv)+z_l(z_ku\ast v)+z_{k+l}(u\ast v)$ for any words $u,v\in\ha^1$ and any positive integers $k,l$ with $\QQ$-bilinearity. We can equip $\ha^1$ with a commutative $\QQ$-algebra structure and denote it by $\ha^1_\ast$. Let $\ha^0$ be the subspace of $\ha^1$ spanned by all words not starting in $z_1$. The linear map defined by
\begin{align}\label{eq:zeta evaluate}
    \zeta:\ha^0&\longrightarrow\mz\\
    z_{k_1}\cdots z_{k_r}&\longmapsto\zeta(k_1,\dots,k_r)
\end{align}
is an algebra homomorphism from $\ha^0_\ast$ to $\mz$. This homomorphism can be extended to a homomorphism $\zeta^\ast:\ha^1_\ast\to\mz\lbrack T\rbrack$, that is, we obtain elements $\zeta^\ast(k_1,\dots,k_r;T)\in\RR\lbrack T\rbrack$ for all $k_1,\dots,k_r\ge1$ called the \emph{harmonic regularized multiple zeta values}. In the case $k_1\ge2$, these coincide with the original multiple zeta values.
To describe the second product, we identify $\ha^1$ with $\QQ+\QQ\langle x,y\rangle y$, the subspace of the non-commutative polynomial ring in two variables $x,y$, where the identification is given by $z_{k_1}\cdots z_{k_r}=x^{k_1-1}y\cdots x^{k_r-1}y$. Then, the shuffle product on $\ha^1$ is defined by $w\shuffle 1=1\shuffle w=w$ for any word $w\in\ha^1$ and $au\shuffle bv=a(u\shuffle bv)+b(au\shuffle v)$ for any words $u,v\in\ha^1$ and any letters $a,b\in\{x,y\}$ with $\QQ$-bilinearity. Due to the iterated integral representation of multiple zeta values, one obtains that the map \eqref{eq:zeta evaluate} gives an algebra homomorphism from $\ha^0_\shuffle$ to $\mz$. There is also a unique extension of the map $\zeta$ to an algebra homomorphism $\zeta^\shuffle:\ha^1_\shuffle\to\mz\lbrack T\rbrack$. 

There are various families of relations, which conjecturally describe all relations among multiple zeta values, i.e., there exist explicit descriptions of the kernel $\ker(\zeta)$ of the map \eqref{eq:zeta evaluate}. One prominent family of such relations are the double shuffle relations (\cite{IKZ}). Define for $u,v \in \ha^1$ the following element in $\ha^1$
\begin{align}\label{eq:defds}
    \Ds(u,v)\coloneqq u\ast v-u\shuffle v.
\end{align}

Then, by the above discussion, for $u,v \in \ha^0$ we get the finite double shuffle relations $\Ds(u,v) \in \ker(\zeta)$. The extended double shuffle relations extend these to $u\in \ha^0 , v\in \ha^0 \cup \{ z_1 \}$. Conjecturally $\ker(\zeta)$ is then the ideal generated by these $\Ds(u,v)$. Further, we have a dimension conjecture by Zagier, which states that
\begin{align}\label{eq:zagierconjagain}
       \sum_{k\geq 0} \dim_\QQ \mz_k X^k  \overset{?}{=} \frac{1}{1-X^2-X^3},
\end{align}
where $\mz_k$ denotes the space of multiple zeta values of weight $k$. While this remains a conjecture for classical multiple zeta values, the analogous statement was proven by Brown (\cite{Br}) for the graded algebra of motivic multiple zeta values $\mz^\mathfrak{m}$.

To make more sense of the analogous conjecture for multiple Eisenstein series, we give a different interpretation of the above formula. For this, consider the series
\begin{align}
    \mathsf{E}(X) = \frac{1}{1-X^2} = 1 + X^2 + X^4 + \dots, \quad \mathsf{O}(X) = \frac{X^3}{1-X^2} = X^3 + X^5 + \dots.
\end{align}
Then one can directly verify that \eqref{eq:zagierconjagain} can be rewritten as
\begin{align}\label{eq:zagierconjagainagain}
\sum_{k\geq 0} \dim_\QQ \mz_k X^k  \overset{?}{=} \frac{1}{1-X^2-X^3} =  \mathsf{E}(X) \cdot  \frac{1}{1 -  \mathsf{O}(X)}.
\end{align}
From the perspective of Hilbert--Poincar\'e series for graded algebras, this reflects the structure of motivic multiple zeta values established by Brown:
\begin{align}\label{eq:mzvalgebraconj}
\mz^\mathfrak{m} \cong \QQ[f_2] \otimes \QQ\langle f_3,f_5,\dots \rangle \cong\QQ[f_2] \otimes U(\mathfrak{g}^\mathfrak{m})^{\vee}
\end{align}
where the isomorphism on the right factor $\QQ\langle f_3,f_5,\dots \rangle$ is equipped with the shuffle product. The left factor $\QQ[f_2]$ corresponds to all even single zeta values (which, by Euler, satisfy $\zeta(2m) \in \QQ[\zeta(2)]$). The right factor can be identified with the universal enveloping algebra of the motivic Lie algebra $\mathfrak{g}^\mathfrak{m} \cong \operatorname{Lie}(\sigma_3,\sigma_5,\sigma_7,\ldots)$, a free Lie algebra with one generator in each odd weight.

\subsection{Multiple Eisenstein series}

Multiple zeta values and modular forms are connected in various ways. Zagier \cites{Za1,Za2} initially noted that there exist some relations among double zeta values whose coefficients originate from modular forms for $\SL_2(\ZZ)$. A formulation of such modular relations for double zeta values was established by Gangl, Kaneko and Zagier \cite{GKZ}. In this work, they also introduced \emph{double Eisenstein series} and constructed regularized double Eisenstein series, which satisfy the extended double shuffle relations. At the same time, Kaneko \cite{Ka} described the modular relation for double zeta values via double Eisenstein series. For general depth $r\ge1$, the multiple Eisenstein series are defined for integers $k_1,\dots,k_r\ge2$ and $\tau\in\HH=\{z\in\CC\mid\Im(z)>0\}$ by
\begin{align}
    G_{k_1,\dots,k_r}(\tau)\coloneqq\lim_{M\to\infty}\lim_{N\to\infty}\summ{\lambda_1,\dots,\lambda_r\in\ZZ_M\tau+\ZZ_N\\\lambda_1\succ\cdots\succ\lambda_r\succ0}\frac{1}{\lambda_1^{k_1}\cdots\lambda_r^{k_r}},
\end{align}
where $\ZZ_M=\{m\in\ZZ\mid|m|<M\}$ and the order $\succ$ on the lattice $\ZZ\tau+\ZZ$ is given by the lexicographical order defined by $m_1\tau+n_1\succ m_2\tau+n_2$ iff $m_1>m_2$ or $m_1=m_2\;\wedge\;n_1>n_2$. Since $G_{k_1,\dots,k_r}(\tau+1)=G_{k_1,\dots,k_r}(\tau)$, which can be obtained by the above definition, the multiple Eisenstein series has a Fourier expansion
\begin{align}\label{eq:mesfourier}
    G_{k_1,\dots,k_r}(\tau)=\zeta(k_1,\dots,k_r)+\sum_{n\ge1}a_nq^n\qquad(a_n\in\mz\lbrack2\pi i\rbrack,\;q=e^{2\pi i\tau}),
\end{align}
which was calculated in depth $r=2$ in \cite{GKZ} and for arbitrary depth by the first author in \cite{Ba1}. More precisely, the Fourier expansion can be written as an explicit $\mz$-linear combination of the $q$-series
\begin{align}
    g_{k_1,\dots,k_r}(\tau)\coloneqq\frac{(-2\pi i)^{k_1+\cdots+k_r}}{(k_1-1)!\cdots(k_r-1)!}\summ{c_1>\cdots>c_r>0\\d_1,\dots,d_r>0}d_1^{k_1-1}\cdots d_r^{k_r-1}q^{c_1 d_1+\cdots+c_r d_r}
\end{align}
for $k_1,\dots,k_r\ge1$. For example, we have 
\begin{align}
G_{4,2}(\tau) &= \zeta(4,2) + 2 \zeta(2)\, g_4(\tau) + 2 \zeta(3) \,g_3(\tau) + 4 \zeta(4) \,g_{2}(\tau) + g_{4,2}(\tau).
\end{align}
The explicit formula for the Fourier expansion will be reviewed in \Cref{thm:Fourier exp}.
We denote the $\QQ$-linear space spanned by all multiple Eisenstein series by
\begin{align}
    \mes\coloneqq\Span_\QQ\{G_{k_1,\dots,k_r}(\tau),G_\varnothing(\tau)\mid r\ge1,k_1,\dots,k_r\ge2\},
\end{align}
where $G_{\varnothing}(\tau)\coloneqq1$ for the \emph{empty index} $\varnothing$. Since multiple Eisenstein series satisfy the harmonic product formula by definition, we know that $\mes$ is a $\QQ$-algebra. The subspace $$\ha^{\ge2}\coloneqq\QQ\langle z_k\mid k\ge2\rangle \quad \subset \quad  \ha^0 \quad  \subset \quad \ha^1$$ is closed under the harmonic product $\ast$. The space $\mes$ can be viewed as the image of the $\QQ$-linear map 
\begin{align}\label{eq:mapG}
    G:\ha^{\ge2} &\longrightarrow  \mathcal{O}(\HH),\\
    z_{k_1}\cdots z_{k_r}&\longmapsto G_{k_1,\dots,k_r}(\tau).
\end{align}
This map is an algebra homomorphism with respect to the harmonic product $\ast$. Similar to the case of multiple zeta values, the space $\mes$ is expected to be graded by weight.

\begin{conj}\label{conj:mesgraded}
We have
$$\mes = \bigoplus_{k\ge0}\mes_k,$$
where $\mes_0\coloneqq\QQ\cdot G_\varnothing(\tau)=\QQ$ and for $k\geq 1$
$$\mes_k\coloneqq\Span_\QQ\{G_{k_1,\dots,k_r}(\tau)\mid k_1,\dots,k_r\ge2,\, k_1+\cdots+k_r=k\}.$$
\end{conj}
Analogous to the case of multiple zeta values, there exists a dimension conjecture for $\mes_k$. This conjecture is based on corresponding conjectures for the associated graded spaces of $q$-analogues of multiple zeta values given by Okounkov in \cite{O} and by the first author and K\"uhn in \cite{BK}, combined with the expectation that multiple Eisenstein series satisfy the same relations as these $q$-analogues modulo lower weight terms as mentioned in \cite{BT}, \cite{BIM} and \cite{HST}. It can be viewed as an analogue of Zagier's conjecture \eqref{eq:zagierconjagainagain}. To state the conjecture, let
\begin{align}
\mathsf{M}(X) =\sum_{k\geq 0} \dim \mathcal{M}_k X^k = \frac{1}{(1-X^4)(1-X^6)},\qquad
\mathsf{S}(X) = \sum_{k\geq 0} \dim \mathcal{S}_k X^k =\frac{X^{12}}{(1-X^4)(1-X^6)}
\end{align}
denote the generating series of the dimensions of modular forms and cusp forms for $\SL_2(\ZZ)$, respectively.

\begin{conj}\label{conj:dimension}
We have
\begin{align}
\sum_{k\ge0}\dim_\QQ \mes_kX^k &= \mathsf{M}(X) \cdot \frac{1}{1 - X^2 - \mathsf{O}(X) + 2 \mathsf{S}(X)} \\
&= \frac{1}{1-X^2-X^3-X^4-X^5+X^8+X^9+X^{10}+X^{11}+X^{12}}.
\end{align}
\end{conj}

The first equation in this conjecture suggests a structural analogue of the motivic isomorphism \eqref{eq:mzvalgebraconj}. Here, the ring of even single zeta values is replaced by the algebra of modular forms $\mathcal{M}$, leading to the expectation that there exists a Lie algebra $\mathfrak{E}$ such that
$$\mes \overset{?}{\cong} \mathcal{M} \otimes U(\mathfrak{E})^{\vee}.$$

The Lie algebra $\mathfrak{E}$ resembles the motivic Lie algebra $\mathfrak{g}^\mathfrak{m}$ discussed above, as it is generated by elements $\sigma_3, \sigma_5, \dots$ in odd weights. However, the dimension formula suggests two essential differences in the structure of $\mathfrak{E}$. First, the term $-X^2$ in the denominator corresponds to the existence of an additional generator $\delta_2$ in weight $2$. Second, the term $+2\mathsf{S}(X)$ indicates that $\mathfrak{E}$ is not free, but subject to relations arising from cusp forms. This Lie algebra is expected to be isomorphic to a Lie subalgebra of the $\sigma$-equivariant derivations discussed in \cite{BIM}, where the additional generator $\delta_2$ corresponds exactly to the derivation of weight $-2$. This derivation plays a central role in the $\mathfrak{sl}_2$-algebra structure conjectured in \cite{BIM} for the formal space, described below in \Cref{conj:sl_2} for the space $\mes$. We will not focus on these aspects in this work, but we expect that the Lie algebra $\mathfrak{E}$ is, for the above reasons, deeply connected to the Lie algebra of derivations of the $\mathfrak{sl}_2$-algebra $\widetilde{\mes^f}$ in \Cref{thm:main2}.

\Cref{conj:dimension} predicts that there exist many linear relations among multiple Eisenstein series. The number of generators of $\mes_k$, which equals the number of indices of weight $k$ with all entries $\ge2$, is given by the $(k-1)$-th Fibonacci number $F_{k-1}$. The conjectured number of independent linear relations is given by the following table.
\begin{align}
\begin{array}{c|ccccccccccccccccc}
\text{weight }k & 0 & 1 & 2 & 3 & 4 & 5 & 6 & 7 & 8 & 9 & 10 & 11 & 12 & 13 & 14 & 15 & 16
\\\hline
\text{$\#\text{ of generators}=F_{k-1}$} & 1 & 0 & 1 & 1 & 2 & 3 & 5 & 8 & 13 & 21 & 34 & 55 & 89 & 144 & 233 & 377 & 610
\\
\text{$\dim_\QQ\mes_k\overset{?}{=}$} & 1 & 0 & 1 & 1 & 2 & 3 & 4 & 7 & 9 & 15 & 21 & 32 & 47 & 70 & 104 & 153 & 228
\\
\text{$\#\text{ of relations}\overset{?}{=}$} & 0 & 0 & 0 & 0 & 0 & 0 & 1 & 1 & 4 & 6 & 13 & 23 & 42 & 74 & 129 & 224 & 382
\end{array}
\end{align}

\begin{ex}\label{ex:wt6and7}
    The relations among multiple Eisenstein series in weight $6$ and $7$ are
\begin{align}
    6G_{3,3}(\tau)-3G_{4,2}(\tau)-G_6(\tau)&=0,\\
    4G_{3,4}(\tau)+3G_{4,3}(\tau)-2G_{5,2}(\tau)-G_7(\tau) &=0.
\end{align}
\end{ex}

In contrast to multiple zeta values, where we have several explicit conjectures describing $\ker(\zeta)$, there were no explicit conjectures yet for the kernel of the map $G$ in \eqref{eq:mapG}. 

For multiple zeta values, the comparison of the two different regularizations plays a crucial role. In the case of multiple Eisenstein series, this perspective is also important and raises the natural question of whether there exist algebra homomorphisms $G^\bullet:\ha^1_\bullet\to\mathcal{O}(\HH)$, for $\bullet\in\{\ast,\shuffle\}$, such that $G^\bullet\mid_{\ha^{\ge2}}=G$ and such that $G^\bullet_{k_1,\dots,k_r}(\tau)$ has a Fourier expansion with a constant term $\zeta^\bullet(k_1,\dots,k_r)\coloneqq\zeta^\bullet(k_1,\dots,k_r;0)$ for any $k_1,\dots,k_r\ge1$. In the case $\bullet=\shuffle$, the first author and Tasaka \cite{BT} constructed the shuffle regularized multiple Eisenstein series $G^\shuffle$ by revealing the correspondence between the Fourier expansion of multiple Eisenstein series and the Goncharov coproduct, which arose in a certain Hopf algebra. Similarly, the harmonic regularized multiple Eisenstein series $G^\ast$ is constructed by the first author in \cite{Ba2}. We review the precise construction of $G^\bullet$ in \Cref{sec:regularization} by using generating functions.

Using the shuffle regularization $G^\shuffle$, one can obtain linear relations among multiple Eisenstein series, called the \emph{restricted double shuffle relation}.
Recall that $\Ds(u,v)\coloneqq u\ast v-u\shuffle v$ for $u,v\in\ha^1$.

\begin{thm}[\cite{BT}]\label{thm:restricted DSR}
    For $u,v\in\ha^{\ge2}$ we have $\Ds(u,v)\in\ker G^\shuffle$.
\end{thm}
This theorem gives a lot of relations among multiple Eisenstein series, but it was already observed in \cite{BT} that it does not suffice to give all relations. 

\subsection{The first two main theorems}\label{subsec:main theorem AB}

The first main theorem is a generalization of \Cref{thm:restricted DSR}, which is obtained by a comparison of the harmonic and shuffle regularization of multiple Eisenstein series. We denote by $\halm$ the space spanned by words $z_{k_1}\cdots z_{k_r}\in\ha^1$ where the indices are \emph{almost} all $\ge 2$; specifically, we allow at most one $z_1$, provided it is not at the beginning of the word. Explicitly,
\begin{align}
    \halm \coloneqq \ha^{\ge2} \oplus \Span_\QQ \{z_{k_1}\cdots z_{k_r} \mid r\ge2, k_1,\dots, k_{i-1}, k_{i+1}, \dots, k_r \ge 2, k_i=1\;(2\le i\le r)\}.
\end{align}
In \Cref{sec:regularization}, we will show that $G^\shuffle$ coincides with $G^\ast$ on the space $\halm$, which as a consequence extends \Cref{thm:restricted DSR} and allows one of the words to be in $\halm$ instead of $\ha^{\ge2}$.
\begin{mainthm}\label{thm:comparison of regularization}
\begin{enumerate}[(i)]
    \item    For $w\in\halm$, we have $$G^\shuffle(w)=G^\ast(w).$$
    \item For $u\in\ha^{\ge2}$ and $v\in\halm$, we have $$\Ds(u,v)\in\ker G^\shuffle.$$
\end{enumerate}
 
\end{mainthm}
\newcommand{\dsht}{\varphi}
As a consequence, we can obtain a family of linear relations among classical multiple Eisenstein series, which play an important role throughout the discussion of this paper. We define the $\QQ$-linear map $\dsht:\ha^1\to\ha^1$ and $\QQ$-bilinear map $R:\ha^1\times\ha^1\to\ha^1$ by
\begin{align}
    \dsht(w)&\coloneqq \Ds(w,z_2)=w\ast z_2-w\shuffle z_2,\\
    R(u,v)&\coloneqq \dsht(u\ast v)-\dsht(u)\ast v-u\ast \dsht(v).
\end{align}
Notice that $R$ measures the failure of $\dsht$ being a derivation with respect to the harmonic product $\ast$. This is motivated by the work \cite{BB} and \cite{BIM} where the operator $\dsht$ is a derivation on a subspace of the algebra of formal multiple Eisenstein series. A consequence of our result is indeed that it gives elements in the kernel of $G$.

By the restricted double shuffle relation (\Cref{thm:restricted DSR}), $\dsht(w)\in\halm$ $(w\in\ha^{\ge2})$ give relations among shuffle regularized multiple Eisenstein series $G^\shuffle$, i.e. we have $\dsht(w)\in\ker G^\shuffle\cap\halm$ for any $w\in\ha^{\ge2}$. By \Cref{thm:comparison of regularization}, one obtains $\ker G^\shuffle\cap\halm=\ker G^\ast\cap\halm$. Therefore, we know that $R(u,v)$ give relations among harmonic regularized multiple Eisenstein series $G^\ast$ as one can show that $R(u,v)$ lies in $\ha^{\ge2}$ for any $u,v\in\ha^{\ge2}$ (\Cref{lem:R in ha2}). As a consequence we obtain:

\begin{mainthm}\label{cor:R}
    For $u,v\in\ha^{\ge2}$, we have $R(u,v)\in\ker G$.
\end{mainthm}

\begin{ex}
    The first two relations given in \Cref{ex:wt6and7} are obtained by this theorem since
    \begin{align}
        R(z_2,z_2)&=6z_3z_3-3z_4z_2-z_6,\\
        R(z_3,z_2)&=4z_3z_4+3z_4z_3-2z_5z_2-z_7.
    \end{align}
\end{ex}
\Cref{cor:R} gives all relations among multiple Eisenstein series up to weight $16$ (but fails in weight 17), assuming \Cref{conj:dimension} is correct. To get the missing relations it seems to be necessary to consider the $\sltwo$-algebra structure for multiple Eisenstein series.

\subsection{$\sltwo$-algebra structure of $\mes$}
An \emph{$\sltwo$-algebra} is an algebra $A$ equipped with a Lie algebra homomorphism $\sltwo \rightarrow \operatorname{Der}(A)$. Equivalently, it can be defined as an algebra $A$ paired with three derivations $W,D,\delta \in \operatorname{Der}(A)$ that satisfy the commutator relations
\begin{align}
[W, D] = 2D, \quad [W, \delta] = -2\delta, \quad [\delta, D] = W.
\end{align}
These derivations $(W,D,\delta)$ form an \emph{$\mathfrak{sl}_2$-triple}, and the kernel $\ker(\delta)$ inherits the structure of a Rankin-Cohen algebra as introduced in \cite{Za4}.

Furthermore, the algebra $\mes$ contains the $\QQ$-algebra of quasi-modular forms, $\widetilde{\mathcal{M}}^{\QQ}=\QQ\lbrack G_2,G_4,G_6\rbrack$, as a subalgebra. It is known (\cite{Za3}) that $\widetilde{\mathcal{M}}^\QQ$ admits an $\sltwo$-algebra structure, and its corresponding Rankin-Cohen algebra is the algebra of modular forms $\mathcal{M}=\QQ[G_4,G_6] = \ker(\delta)$, which is closed under the classical Rankin-Cohen brackets. Building on this framework, the following conjecture from \cite{BIM} proposes a natural generalization of the $\sltwo$-algebra $\widetilde{\mathcal{M}}^\QQ$.
\begin{conj}[\cite{BIM}]\label{conj:sl_2}
    \begin{enumerate}
        \item The maps $W^\prime ,D^\prime,\delta^\prime$ defined on the generators of $\mes$ by
        \begin{align}
            W^\prime&:G_{k_1,\dots,k_r}(\tau)\longmapsto(k_1+\cdots +k_r)G_{k_1,\dots,k_r}(\tau),\\
            D^\prime&:G_{k_1,\dots,k_r}(\tau)\longmapsto 2\pi i\frac{d}{d\tau}G_{k_1,\dots,k_r}(\tau),\\
            \delta^\prime&:G_{k_1,\dots,k_r}(\tau)\longmapsto \begin{cases}-\frac{1}{2}G_{k_2,\dots,k_r}(\tau)&k_1=2,\\0&k_1>2,\end{cases}
        \end{align}
        give well-defined $\QQ$-linear maps $\mes\to\mes$.
        \item The maps $W^\prime,D^\prime,\delta^\prime$ are derivations on $\mes$.
        \item $(W^\prime,D^\prime,\delta^\prime)$ forms an $\sltwo$-triple, i.e. we have the commutator relations
        \begin{align}
            \lbrack W^\prime,D^\prime\rbrack=2D^\prime,\quad\lbrack W^\prime,\delta^\prime\rbrack=-2\delta^\prime,\quad\lbrack \delta^\prime,D^\prime\rbrack=W^\prime
        \end{align}
        and thus $\mes$ is an $\sltwo$-algebra.
    \end{enumerate}
\end{conj}
Notice that even the well-definedness in (i) is not clear, as, for example, the map $W^\prime$ being well-defined is equivalent to \Cref{conj:mesgraded}. As a consequence of our fourth main theorem (\Cref{thm:der}), the map $D^\prime$ is well-defined, i.e. the space $\mes$ is closed under the derivative $2\pi i\frac{d}{d\tau}$. Since $D^\prime$ is a differential operator, it is then automatically a derivation on $\mes$. The well-definedness of $W^\prime$ and $\delta^\prime$ remains open (see \Cref{sec:another formal MES} for a discussion of $\delta^\prime$).
The restriction of the maps $W^\prime,D^\prime,\delta^\prime$ to $\widetilde{\mathcal{M}}^\QQ$ are well-defined derivations and form an $\sltwo$-triple.

\subsection{Formal multiple Eisenstein series and the third main theorem}

In \cite{BIM}, the authors introduced the notion of \emph{formal multiple Eisenstein series} and studied their derivations and the $\sltwo$-algebra structure. The algebra of formal multiple Eisenstein series $\bmes^f$ is defined (\Cref{def:formal MES}) to be the $\QQ$-linear space spanned by symbols $G^f\binom{k_1,\dots,k_r}{d_1,\dots,d_r}$ ($k_1,\dots,k_r\geq1,d_1,\dots,d_r\geq0$), which are swap invariant (\Cref{def:swap}) and whose product is given by (a double-indexed version of) the harmonic product. Here, swap invariance is a family of relations among double-indexed objects, which can be interpreted via the conjugation of partitions. In the case of depth $1$, this just means $G^f\binom{k}{d}=\frac{d!}{(k-1)!}G^f\binom{d+1}{k-1}$. The use of double indices is necessary to make sense of the swap invariance. But often we will be interested in the case when $d_1=\cdots=d_r=0$ and denote
\begin{align}
    G^f(k_1,\dots,k_r)\coloneqq G^f\binom{k_1,\dots,k_r}{0,\dots,0}.
\end{align}
Let $\bmes^{f,0}$ (resp. $\mes^f$) be the space spanned by all $G^f(k_1,\dots,k_r)$ with $k_1,\dots,k_r\geq1$ (resp. $k_1,\dots,k_r\geq2$). Then, the space $\bmes^f$ is expected to coincide with the subspace $\bmes^{f,0}$ (\cite{BIM}), and the space $\mes^f$ can be seen as the formal version of the algebra $\mes$. One of the main results in \cite{BIM} is that the algebra $\bmes^f$ has an $\sltwo$-algebra structure. The $\sltwo$-triple is denoted by $(W,D,\delta)$ (the definition is given in \Cref{subsec:formal MES}). In particular, the derivation $D$ is defined by
\begin{align}
    D\left(G^f\binom{k_1,\dots,k_r}{d_1,\dots,d_r}\right)\coloneqq\sum_{j=1}^rk_jG^f\binom{k_1,\dots,k_j+1,\dots,k_r}{d_1,\dots,d_j+1,\dots,d_r}.
\end{align}
In \cite{BIM} it was shown that $D$ coincides with $\dsht$ within $\bmes^{f,0}$, i.e., we have $D(G^f(k_1,\dots,k_r))=G^f(\dsht(z_{k_1}\cdots z_{k_r}))$ for any $k_1,\dots,k_r\geq1$. From this fact, we know that $\bmes^{f,0}$ is an $\sltwo$-subalgebra of $\bmes^f$, and for any $u,v\in\ha^1$, $R(u,v)$ gives a relation among formal multiple Eisenstein series. In \cite{BIM}, the authors conjectured that $\mes^f$ is also an $\sltwo$-subalgebra, which reduces to the question of whether $\mes^f$ is closed under $D$; this question remained open. Another problem was that a realization map from $\mes^f$ to $\mes$ had not been constructed (despite the conjecture that $\mes^f$ is isomorphic to $\mes$). This was due to the fact that, while the first author and Burmester (\cite{BB}) constructed a realization from $\bmes^f$ to $\QQ\llbracket q\rrbracket$ (called the combinatorial multiple Eisenstein series), a realization from $\bmes^f$ to $\mathcal{O}(\HH)$ had not been constructed before. The third main result in this work provides such a realization; that is, we define the bi-multiple Eisenstein series $\bMES\binom{k_1,\dots,k_r}{d_1,\dots,d_r}(\tau)\in\mathcal{O}(\HH)$ for $k_1,\dots,k_r\geq1$, $d_1,\dots,d_r\geq0$, which recover $G_{k_1,\dots,k_r}(\tau)$ and satisfy the swap invariance and the harmonic product.
\begin{mainthm}\label{thm:biMES}
    There exists an algebra homomorphism $\bMES:\bmes^f\rightarrow\mathcal{O}(\HH):G^f(w)\mapsto\bMES(w)$ such that
    \begin{enumerate}
        \item $\bMES\binom{k_1,\dots,k_r}{0,\dots,0}(\tau)=G_{k_1,\dots,k_r}(\tau)$ for any $k_1,\dots,k_r\geq2$.
        \item $\bMES(D(f))=2\pi i\frac{d}{d\tau}\bMES(f)$ for any $f\in\bmes^f$.
    \end{enumerate}
\end{mainthm}

\subsection{The fourth main theorem}

Hirose, Maesaka, Seki and Watanabe (\cite{HMSW}) provided an explicit algorithm to express every multiple zeta value as a $\ZZ$-linear combination of $\zeta(k_1,\dots,k_r)$ where $k_1,\dots,k_r\ge2$. Let $\mathcal{D}:\ha^0\to\ha^{\ge2}$ be the Drop1 operator introduced in \cite{HMSW} (see \Cref{def:drop1operator}) and define the $\QQ$-linear map $\der:\ha^0\to\ha^{\ge2}$ by $\der\coloneqq-\mathcal{D}\circ \dsht$. We give precise definitions in \Cref{subsec:Drop1}. The fourth main theorem of this work is the following explicit formula for the derivative of (formal) multiple Eisenstein series.
\begin{mainthm}\label{thm:der}
    \begin{enumerate}
        \item For $k_1,\dots,k_r\geq2$, we have
        \begin{align}
            D(G^f(k_1,\dots,k_r))=G^f(\der(z_{k_1}\cdots z_{k_r})),
        \end{align}
        and thus $\mes^f$ is an $\sltwo$-subalgebra of $\bmes^f$.
        \item For $k_1,\dots,k_r\geq2$, we have 
        \begin{align}
            2\pi i\frac{d}{d\tau}G_{k_1,\dots,k_r}(\tau)&=G(\der(z_{k_1}\cdots z_{k_r})).
        \end{align}
        In particular, the space $\mes$ is closed under $D^\prime=2\pi i\frac{d}{d\tau}$.
    \end{enumerate}
\end{mainthm}

As an application of \Cref{thm:der}, together with the combinatorial multiple Eisenstein series introduced by the first author and Burmester in \cite{BB}, we prove the following prediction of Okounkov (\cite[Conjecture 1]{O}) on the space $\qmzv$ spanned by his $q$-analogues of multiple zeta values (see \Cref{subsec:okounkov} for the definition of $\qmzv$).
\begin{cor}\label{cor:okounkov}
    The space $\qmzv$ is closed under the operator $q\frac{d}{dq}$.
\end{cor}

\subsection{The fifth main theorem}

As mentioned in \Cref{subsec:main theorem AB}, the relations obtained from \Cref{cor:R} do not capture all relations among multiple Eisenstein series. However, by \Cref{thm:der}, one obtains further relations by differentiating them. Let $\mathsf{DR}_\ast\coloneqq\Span_\QQ\{\der^n(R(u,v))\ast w\mid n\geq0,u,v,w\in\ha^{\geq2}\}$. By \Cref{cor:R} and \Cref{thm:der}, it holds $\mathsf{DR}_\ast\subset\ker G$. We conjecture that this family gives a complete description of all relations among multiple Eisenstein series.
\begin{conj}\label{conj:all relation}
    We have $\mathsf{DR}_\ast=\ker G$.
\end{conj}
Based on this conjecture, we define another formal multiple Eisenstein space as 
\begin{align}
    \widetilde{\mes^f}\coloneqq\quotient{\ha^{\ge2}_\ast} {\mathsf{DR}_\ast}.
\end{align}
We denote by $\widetilde{G^f}(k_1,\dots,k_r)$ the class of $z_{k_1}\cdots z_{k_r}$ in $\widetilde{\mes^f}$. Assuming \Cref{conj:all relation}, the space $\widetilde{\mes^f}$ is isomorphic to the multiple Eisenstein space $\mes$. According to \Cref{conj:sl_2}, $\mes$ is expected to have an $\sltwo$-algebra structure, and \Cref{thm:der} shows that the formal multiple Eisenstein algebra $\mes^f$ is an $\sltwo$-algebra. The fifth main theorem of this paper states that this is also the case for the alternative formal multiple Eisenstein algebra $\widetilde{\mes^f}$.
\begin{mainthm}\label{thm:main2}
    $\widetilde{\mes^f}$ is an $\sltwo$-algebra.
\end{mainthm}

\subsection{Organization}
The organization of this paper is as follows. In \Cref{sec:regularization}, we review the construction of two regularizations of multiple Eisenstein series developed in \cite{BT}, \cite{Ba2}. We compare them and prove the first main theorem. In \Cref{sec:bi MES}, we recall the notion of formal multiple Eisenstein series introduced in \cite{BIM} and define the bi-multiple Eisenstein series as a realization of the formal multiple Eisenstein series. In \Cref{sec:der}, we provide a derivative formula, prove the fourth main theorem, and, as an application, prove \Cref{cor:okounkov} on Okounkov's $q$-analogues of multiple zeta values. In \Cref{sec:another formal MES}, we define another formal multiple Eisenstein space, prove the fifth main theorem, and introduce the multiple Eisenstein-diamond series.
\section*{Acknowledgments}
The authors would like to thank Hanamichi Kawamura for his valuable comments. The authors would also like to thank Shin-ichiro Seki for comments and corrections on an early version of this draft. This project was partially supported by JSPS KAKENHI Grant 23K03030 (H.B.) and 26KJ0525 (H.K).

\section{Regularization of multiple Eisenstein series}\label{sec:regularization}

\subsection{Mould theory}

For the purpose of comparing the two regularizations, it is convenient to use the language of \emph{moulds} introduced by \'Ecalle in \cite{Ec}.
\begin{dfn}
    Let $R$ be a $\QQ$-algebra. Define the set of bimoulds over $R$ valued in formal power series by
    \begin{align}
        \BIMU(R)\coloneqq\left\{M=\left(M_r\binom{u_1,\dots,u_r}{v_1,\dots,v_r}\right)_{r\ge0}\in\prod_{r\ge0}R\llbracket u_1,v_1,\dots,u_r,v_r\rrbracket\right\}.
    \end{align}
\end{dfn}
We often omit $R,r$ in $\BIMU(R),M_r$ and simply write $\BIMU,M$, respectively.
\begin{enumerate}[(i)]
    \item We call $M\in\BIMU$ a \emph{constant mould} if $M\binom{u_1,\dots,u_r}{v_1,\dots,v_r}\in R$ for every $r\ge0$.
    \item We call $M\in\BIMU$ a \emph{$u$-mould} if $M\binom{u_1,\dots,u_r}{v_1,\dots,v_r}$ does not depend on $v_1,\dots,v_r$ for any $r\ge0$.
    \item We call $M\in\BIMU$ a \emph{$v$-mould} if $M\binom{u_1,\dots,u_r}{v_1,\dots,v_r}$ does not depend on $u_1,\dots,u_r$ for any $r\ge0$.
    \item Define the multiplication $\times:\BIMU\times\BIMU\to\BIMU$ by
    \begin{align}
        (M\times N)(w_1,\dots,w_r)=\sum_{i=0}^rM(w_1,\dots,w_i)N(w_{i+1},\dots,w_r),
    \end{align}
    where we denote $w_i=\binom{u_i}{v_i}$.
    \item Define maps $\mathsf{neg},\anti,\pari,\swap:\BIMU\to\BIMU$ by
    \begin{align}
        \mathsf{neg}(M)(w_1,\dots,w_r)&\coloneqq M(-w_1,\dots,-w_r),\\
        \mathsf{anti}(M)(w_1,\dots,w_r)&\coloneqq M(w_r,\dots,w_1),\\
        \mathsf{pari}(M)(w_1,\dots,w_r)&\coloneqq(-1)^rM(w_1,\dots,w_r),\\
        \swap(M)\binom{u_1,\dots,u_r}{v_1,\dots,v_r}&\coloneqq M\binom{v_r,v_{r-1}-v_r,\dots,v_1-v_2}{u_1+\cdots+u_r,u_1+\cdots+u_{r-1},\dots,u_1}.
    \end{align}
    \item For $B,C\in\mathsf{BIMU}$, define $\gaxit_{B,C}:\mathsf{BIMU}\to\mathsf{BIMU}$ by
    \begin{align}
        \gaxit_{B,C}(A)(\bm{v})=\summ{\bm{v}=\bm{a}_1\bm{b}_1\bm{c}_1\cdots \bm{a}_s\bm{b}_s\bm{c}_s\\\bm{b}_i,\bm{c}_i\bm{a}_{i+1}\ne\emptyset}A(\lceil\bm{b}_1\rceil\cdots\lceil\bm{b}_s\rceil)B(\bm{a}_1\rfloor)\cdots B(\bm{a}_s\rfloor)C(\lfloor\bm{c}_1)\cdots C(\lfloor\bm{c}_s),
    \end{align}
    where the flexion markers $\lceil,\rceil,\lfloor,\rfloor$ are defined by
    \begin{align}
        \lceil\bm{b}&\coloneqq\binom{u_1+\cdots+u_{i+1},u_{i+2},\dots,u_r}{v_{i+1},v_{i+2},\dots,v_r},&\bm{a}\rceil&\coloneq\binom{u_1,\dots,u_{i-1},u_i+\cdots+u_r}{v_1,v_2,\dots,v_i},\\
        \lfloor\bm{b}&\coloneq\binom{u_i,\dots,u_r}{v_{i+1}-v_i,\dots,v_r-{v_i}},&\bm{a}\rfloor&\coloneqq\binom{u_1,\dots,u_i}{v_1-v_{i+1},\dots,v_i-v_{i+1}}
    \end{align}
    for a composition $\bm{a}\bm{b}=(w_1,\dots,w_i)(w_{i+1},\dots,w_r)=(w_1,\dots,w_r)$ $(0\le i\le r)$.
    \item For a $v$-mould $M$ and $\bm{k}=(k_1,\dots,k_r)\in(\ZZ_{>0})^r$, we denote by $\langle M\mid \bm{k}\rangle=m(k_1,\dots,k_r)$ the coefficient of $v_1^{k_1-1}\cdots v_r^{k_r-1}$ in $M(v_1,\dots,v_r)$, i.e. $$M(v_1,\dots,v_r)=\sum_{k_1,\dots,k_r>0}m(k_1,\dots,k_r)v_1^{k_1-1}\cdots v_r^{k_r-1}.$$
\end{enumerate}
\begin{rem}\label{rem:identification}
    In this paper, we take $R$ to be commutative in all cases except for \Cref{subsec:formal MES}. Given a commutative $\QQ$-algebra $R$ and $\QQ$-linear map $M:\ha^1\to R$, we can identify $M$ with a $v$-mould $M^\prime=(M^\prime(v_1,\dots,v_r))_{r\ge0}$ defined by
    \begin{align}
        M^\prime(v_1,\dots,v_r)\coloneqq\begin{cases}M(1)&r=0,\\\sum_{k_1,\dots,k_r>0}M(z_{k_1}\cdots z_{k_r})v_1^{k_1-1}\cdots v_r^{k_r-1}&r\ge1.\end{cases}
    \end{align}
    Under this identification, we can rephrase some notation (cf. \cite[Section 3]{BB}):
    \begin{enumerate}[(i)]
        \item $M:\ha^1_\ast\to R$ is an algebra homomorphism $\Leftrightarrow$ $M^\prime$ is \emph{symmetril}.
        \item $M:\ha^1_\shuffle\to R$ is an algebra homomorphism $\Leftrightarrow$ $\swap(M^\prime)$ is \emph{symmetral}.
    \end{enumerate}
    Throughout this paper, we identify a map $M:\ha^1\to R$ with a corresponding $v$-mould $M^\prime$, and use the same notation for both.
\end{rem}

\subsection{Goncharov coproduct in mould theory}\label{subsec:Goncharov}
In this section, we reinterpret the Goncharov coproduct in terms of moulds for the comparison of shuffle and harmonic regularizations of multiple Eisenstein series. We note that a corresponding formulation of the coproduct and shuffle regularization was recently given in \cite{HST} using the language of non-commutative power series.

Denote $B_{-}=\mathsf{neg}\circ\anti\circ\pari(B)$ for a bimould $B$, and denote $\gilat_B=\gaxit_{B,B_{-}}$. Define a map $\gila:\BIMU\times\BIMU\to\BIMU$ by
\begin{align}
    \gila(A,B)=\gilat_B(A)\times B.
\end{align}
Hereinafter, we only consider $v$-moulds and take a suitable $\QQ$-algebra $R$ as needed.
\begin{lem}\label{lem:gila}
    For $v$-moulds $A,B$ and $\bm{k}=(k_1,\dots,k_r)\in(\ZZ_{>0})^r$, $\langle\gila(A,B)\mid\bm{k}\rangle$ is equal to
    \begin{align}
        &\summ{0=t_0<q_1\le t_1<q_2\le t_2<\cdots<q_s\le t_s<r+1\\(0\le s\le r)}\summ{n_{t_{j-1}+1},\dots,n_{t_j}\ge1\\n_{t_{j-1};t_j}=k_{t_{j-1};t_j}\\(1\le j\le s)}a(n_{q_1},\dots,n_{q_s})b(k_{t_s+1},\dots,k_r)\\
        &\times\prod_{j=1}^s\bigg\{(-1)^{k_{t_{j-1};t_j}+k_{q_j}+n_{t_{j-1};q_j-1}}\Bigg(\prod_{\substack{p=t_{j-1}+1\\p\ne q_j}}^{t_j}\binom{n_p-1}{k_p-1}\Bigg)b(n_{t_{j-1}+1},\dots,n_{q_j-1})b(n_{t_j},\dots,n_{q_j+1})\bigg\},
    \end{align}
    where $k_{i;j}=k_{i+1}+\cdots+k_j$ for $i<j$.
\end{lem}
\begin{proof}
    This follows from \Cref{lem:gilat}.
\end{proof}
In \cite{Go}, Goncharov considered a formal version of the iterated integrals
\begin{align}
    \int_{a_{k+1}}^{a_0}\frac{dt_1}{t_1-a_1}\cdots\int_{a_{k+1}}^{t_{k-2}}\frac{dt_{k-1}}{t_{k-1}-a_{k-1}}\int_{a_{k+1}}^{t_{k-1}}\frac{dt_k}{t_k-a_k}\qquad(a_i\in\CC).
\end{align}
He proved that the space generated by such formal iterated integrals carries a Hopf algebra structure, whose coproduct is known as \textit{Goncharov's coproduct}. A certain quotient of the space restricted to $a_0,\dots,a_{k+1}\in\{0,1\}$ can be identified with the shuffle algebra $\ha^1_\shuffle$, and we can equip $\ha^1_\shuffle$ with the Hopf algebra structure with the coproduct $\Delta_\mathrm{G}$ (see \cite{BT}). The explicit formula for $\Delta_\mathrm{G}$ is obtained in \cite{BT}.
\begin{prop}[\cite{BT}]\label{prop:explicit Goncharov coproduct}
    For a word $w=z_{k_1}\cdots z_{k_r}\in\ha^1$, $\Delta_{\mathrm{G}}(w)$ is equal to
    \begin{align}
        &\summ{0=t_0<q_1\le t_1<q_2\le t_2<\cdots<q_s\le t_s<r+1\\(0\le s\le r)}\summ{n_{t_{j-1}+1},\dots,n_{t_j}\ge1\\n_{t_{j-1};t_j}=k_{t_{j-1};t_j}\\(1\le j\le s)}(z_{n_{q_1}}\cdots z_{n_{q_s}}\otimes z_{k_{t_s+1}}\cdots z_{k_r})\\
        &\shuffle\bigsha_{j=1}^s\!\bigg\{(-1)^{k_{t_{j-1};t_j}+k_{q_j}+n_{t_{j-1};q_j-1}}\Bigg(\prod_{\substack{p=t_{j-1}+1\\p\ne q_j}}^{t_j}\binom{n_p-1}{k_p-1}\Bigg)(1\otimes z_{n_{t_{j-1}+1}}\cdots z_{n_{q_j-1}}\shuffle z_{n_{t_j}}\cdots z_{n_{q_j+1}})\bigg\}.
    \end{align}
\end{prop}

\begin{rem}\label{rem:gila and Goncharov}
    For a commutative $\QQ$-algebra $R$ and $\QQ$-linear maps $A,B:\ha^1\to R$, the convolution of $A$ and $B$ is defined by $A\star B\coloneqq\mu\circ(A\otimes B)\circ\Delta_\mathrm{G}$, where $\mu$ is the multiplication of $R$. \Cref{lem:gila} and \Cref{prop:explicit Goncharov coproduct} suggest that $\gila(A,B)$ coincides with $A\star B$ under the identification stated in \Cref{rem:identification}. And thus, $\gila(A,B)=A\star B$ is an algebra homomorphism on $\ha^1_\shuffle$ (this is equivalent to saying $\swap(\gila(A,B))$ is symmetral) if $A,B$ are algebra homomorphisms on $\ha^1_\shuffle$.
\end{rem}

\subsection{Shuffle regularization}

In this subsection, we briefly review the construction of the shuffle regularized multiple Eisenstein series $G^\shuffle$ developed in \cite{BT}. The authors used the Fourier expansion, which can be written in terms of the Goncharov coproduct ($\gila$-product). Let $\mathbb{I}^{\ge2}$ be the set of all indices $(k_1,\dots,k_r)$ whose entries are $\ge2$. The following theorem gives the explicit Fourier expansion of multiple Eisenstein series in the convergent case.
\begin{thm}[\cite{Ba1}]\label{thm:Fourier exp}
    For $\bm{k}\in\mathbb{I}^{\ge2}$, we have $G_{\bm{k}}(\tau)=\left\langle\gila(g(\tau),\zeta^\shuffle)\mid\bm{k}\right\rangle$, where $g(\tau)$ and $\zeta^\shuffle$ are $v$-moulds defined as generating series of $g_{\bm{k}}(\tau)$ and $\zeta^\shuffle(\bm{k})$, respectively.
\end{thm}
Note that, in the explicit formula for the Fourier expansion of $G_{\bm{k}}(\tau)$ for $\bm{k}\in\mathbb{I}^{\ge2}$, only admissible multiple zeta values appear. Examples in depth $2$ and $3$ are as follows.
\begin{ex}
    For $k_1,k_2,k_3\ge2$, we have
    \begin{align}
        G_{k_1,k_2}(\tau)&=g_{k_1,k_2}(\tau)+\sum_{\substack{n_1+n_2=k_1+k_2\\n_1,n_2\ge2}}\left(\delta_{n_1,k_1}+\mathsf{b}^{n_2}_{k_1,k_2}\right)g_{n_1}(\tau)\zeta(n_2)+\zeta(k_1,k_2),\\
        G_{k_1,k_2,k_3}(\tau)&=g_{k_1,k_2,k_3}(\tau)+g_{k_1,k_2}(\tau)\zeta(k_3)+g_{k_1}(\tau)\zeta(k_2,k_3)+\zeta(k_1,k_2,k_3)\\
        &\quad+\sum_{\substack{n_1+n_2+n_3=k_1+k_2+k_3\\n_1,n_2,n_3\ge2}}\bigg\{\left(\delta_{n_1,k_1}\mathsf{b}^{n_3}_{k_2,k_3}+\delta_{n_2,k_3}\mathsf{b}^{n_3}_{k_1,k_2}\right)g_{n_1,n_2}(\tau)\zeta(n_3)\\
        &\quad+\left((-1)^{k_3-n_1}\binom{n_2-1}{k_1-1}+(-1)^{k_2+k_3}\binom{n_2-1}{k_3-1}\right)\binom{n_3-1}{k_2-1}g_{n_1}(\tau)\zeta(n_2,n_3)\\
        &\quad+\left((-1)^{k_1+k_3-n_2}\binom{n_2-1}{k_1-1}\binom{n_3-1}{k_3-1}+\delta_{n_3,k_3}\mathsf{b}^{n_2}_{k_1,k_2}\right)g_{n_1}(\tau)\zeta(n_2)\zeta(n_3)\bigg\},
    \end{align}
    where $\mathsf{b}^{n_2}_{k_1,k_2}=(-1)^{n_2-k_1}\binom{n_2-1}{k_1-1}+(-1)^{k_2}\binom{n_2-1}{k_2-1}$.
\end{ex}
As mentioned in \Cref{rem:gila and Goncharov}, whenever $A,B$ satisfy the shuffle product (i.e. if $\swap(A),\swap(B)$ are symmetral), $\gila(A,B)$ also satisfies the shuffle product.
\begin{dfn}[\cite{BT}]\label{dfn:shufflereg}
    The shuffle regularization for $g(\tau)$ is defined by
    \begin{align}
        &\quad g^\shuffle(\tau)(v_1,\dots,v_r)\\
        &\coloneqq\sum_{k_1,\dots,k_r\ge1}g^\shuffle_{k_1,\dots,k_r}(\tau)v_1^{k_1-1}\cdots v_r^{k_r-1}\\
        &\coloneqq\summ{i_1+\cdots+i_m=r\\i_1,\dots,i_m>0}(-2\pi i)^{r-m}H(\tau)\binom{i_1,\dots,i_m}{v_{r-i_1+1},v_{r-i_1-i_2+1}-v_{r-i_1+1},\dots,v_1-v_{r-i_1-\cdots-i_{m-1}+1}},
    \end{align}
    where 
    \begin{align}\label{eq:H}
    H(\tau)\binom{n_1,\dots,n_r}{v_1,\dots,v_r}\coloneqq\frac{(-2\pi i)^r}{n_1!\cdots n_r!}\sum_{d_1>\cdots>d_r>0}\prod_{j=1}^re^{-2\pi id_jv_j}\left(\frac{q^{d_j}}{1-q^{d_{j}}}\right)^{n_j}.
\end{align}
\end{dfn}
\begin{dfn}[\cite{BT}]
    The shuffle regularized multiple Eisenstein series $G^\shuffle$ is defined by $G^\shuffle(\tau)\coloneqq\gila(g^\shuffle(\tau),\zeta^\shuffle)$, i.e. the map $G^\shuffle:\ha^1\to\mathcal{O}(\HH)$ is defined by $G^\shuffle(\tau)=g^\shuffle(\tau)\star\zeta^\shuffle$.
\end{dfn}
\begin{prop}[\cite{BT}]
    \begin{enumerate}
        \item For $\bm{k}\in\mathbb{I}^{\ge2}$, we have $G^\shuffle_{\bm{k}}(\tau)=G_{\bm{k}}(\tau)$.
        \item $G^\shuffle(\tau)$ satisfy the shuffle product $\shuffle$.
    \end{enumerate}
\end{prop}

\subsection{Harmonic regularization}

In this subsection, we review the construction of the harmonic regularized multiple Eisenstein series $G^\ast$ established in \cite{Ba2} via generating functions. A key ingredient in the study of multiple Eisenstein series is the \emph{multitangent function} introduced by Bouillot \cite{Bo}. Every multiple Eisenstein series can be written via multitangent functions.
\begin{dfn}[\cite{Bo}]
    For $k_1,k_r\ge2$ and $k_2,\dots,k_{r-1}\ge1$, define the multitangent function $\Psi_{k_1,\dots,k_r}(z)$  by
    \begin{align}
        \Psi_{k_1,\dots,k_r}(z)\coloneqq\sum_{n_1>\cdots>n_r}\frac{1}{(z+n_1)^{k_1}\cdots(z+n_r)^{k_r}}.
    \end{align}
    When $r=1$, we refer to $\Psi_k(z)$ as the monotangent function and additionally define $\Psi_1(z)\coloneqq\pi\cot(\pi z)$.
\end{dfn}
The series in the right-hand side converges absolutely and defines a holomorphic function on $\CC\setminus\ZZ$.
\begin{lem}\label{lem:decom G}
    For $\bm{k}=(k_1,\dots,k_r)\in\mathbb{I}^{\ge2}$, we have
    \begin{align}
        G_{k_1,\dots,k_r}(\tau)=\sum_{i=0}^r\MEI_{k_1,\dots,k_i}(\tau)\zeta(k_{i+1},\dots,k_r),
    \end{align}
    where the holomorphic function $\MEI_{\bm{k}}(\tau)$ on $\HH$ is defined for $\bm{k}=(k_1,\dots,k_r)\in\mathbb{I}^{\ge2}$ by
    \begin{align}\label{eq:def MEI}
        \MEI_{\bm{k}}(\tau)&\coloneqq\sum_{\substack{w_1\cdots w_s=\bm{k}\\w_i\ne\emptyset}}\sum_{m_1>\cdots>m_s>0}\Psi_{w_1}(m_1\tau)\cdots\Psi_{w_s}(m_s\tau)\\
        &=\lim_{M\to\infty}\lim_{N\to\infty}\summ{\lambda_1,\dots,\lambda_r\in\HH\cap(\ZZ_M\tau+\ZZ_N)\\\lambda_1\succ\cdots\lambda_r\succ0}\frac{1}{\lambda_1^{k_1}\cdots\lambda_r^{k_r}}.
    \end{align}
\end{lem}
In \cite{Bo}, Bouillot constructed the harmonic regularized multitangent functions $\Psi^\ast_{k_1,\dots,k_r}(z)$ and showed that they can be written as $\mz$-linear combinations of monotangent functions. In this paper, we do not present the precise construction; rather, we introduce the reduction to monotangent functions. Define $v$-moulds $\Psi^\ast(z),T(z),\delta$ as follows. 
\begin{align}
    \Psi^\ast(z)(v_1,\dots,v_r)&\coloneqq\sum_{k_1,\dots,k_r>0}\Psi^\ast_{k_1,\dots,k_r}(z)v_1^{k_1-1}\cdots v_r^{k_r-1},\\
    T(z)(v_1,\dots,v_r)&\coloneqq\begin{cases}\sum_{k>0}\Psi_k(z)v_1^{k-1}&r=1\\0&\otherwise\end{cases},\\
    \delta(v_1,\dots,v_r)&\coloneqq\begin{cases}\frac{(\pi i)^r}{r!}&r:\even\\0&r:\odd\end{cases}.
\end{align}
\begin{thm}[\cite{Bo}]\label{thm:reduction}
    The following factorization of mould holds
    \begin{align}
        \Psi^\ast(z)(\bm{v})=\delta(\bm{v})+\sum_{\bm{v}=\bm{a}\bm{b}\bm{c}}\zeta^\ast(\bm{a}\rfloor)T(z)(\bm{b})\zeta^\ast_-(\lfloor \bm{c}).
    \end{align}
\end{thm}
Note that $\Psi_1(z)$ only appears in the reduction of $\Psi^\ast_{\bm{k}}(z)$ when $\bm{k}=\{1\}^{2n+1}$ for some $n\geq0$, and in the case when $\bm{k}=\{1\}^n$, the reduction is given as follows:
\begin{align}\label{eq:reduction ones}
    \Psi^\ast_{\{1\}^{2n}}(z)=\frac{(-1)^n}{(2n)!}\pi^{2n},\qquad\Psi^\ast_{\{1\}^{2n+1}}(z)=\frac{(-1)^n}{(2n+1)!}\pi^{2n}\Psi_1(z).
\end{align}
In \cite{Ba2}, the first author constructed the harmonic regularization for $\MEI(\tau)$ based on the result from Ihara--Kaneko--Zagier \cite{IKZ}. In this paper, we construct the regularization based on the work of \'Ecalle \cite{Ec}. For $M>0$, we define the $v$-moulds $\co{\MEI}_M(\tau),\mathsf{do}{\MEI}_M(\tau)$ over $\mathcal{O}(\HH)$ by
\begin{align}
    \co{\MEI}_M(\tau)(v_1,\dots,v_r)&\coloneqq(-1)^r\sum_{M>m_1\ge\cdots\ge m_r>0}\mu^{m_1,\dots,m_r}\Psi_1(m_1\tau)\cdots\Psi_1(m_r\tau),\\
    \mathsf{do}{\MEI}_M(\tau)(v_1,\dots,v_r)&\coloneqq\sum_{k_1,\dots,k_r\ge1}\MEI^M_{k_1,\dots,k_r}(\tau)v_1^{k_1-1}\cdots v_r^{k_r-1},
\end{align}
where $\mu^{m_1,\dots,m_r}$ and $\MEI^M_{k_1,\dots,k_r}(\tau)$ are defined by
\begin{align}
    \mu^{m_1,\dots,m_r}&=\frac{1}{r_1!\cdots r_s!}\quad((m_1,\dots,m_r)=(\{n_1\}^{r_1},\dots,\{n_s\}^{r_s})),\\
    \MEI^{M}_{k_1,\dots,k_r}(\tau)&=\sum_{\substack{w_1\cdots w_s=(k_1,\dots,k_r)\\w_i\ne\emptyset}}\sum_{M>m_1>\cdots>m_s>0}\Psi^\ast_{w_1}(m_1\tau)\cdots\Psi^\ast_{w_s}(m_s\tau).
\end{align}
We define the $v$-mould $\MEI^\ast(\tau)$ by
\begin{align}
\MEI^\ast(\tau)\coloneqq\lim_{M\to\infty}(\co{\MEI}_M(\tau)\times \mathsf{do}{\MEI}_M(\tau)).
\end{align}
Note that this mould-theoretic construction is based on the work by \'Ecalle \cite{Ec}. In \Cref{sec:appendix}, we show that this construction is essentially the same as the harmonic regularization by Ihara--Kaneko--Zagier \cite{IKZ}. Using \Cref{thm:reduction}, we can decompose the mould $\mathsf{do}{\MEI}_M(\tau)$ as follows.
\begin{lem}\label{lem:doh} We have
    \begin{align}
        \mathsf{do}{\MEI}_M(\tau)(\bm{v})=\sum_{\bm{v}=\underbrace{\bm{a}_1\bm{b}_1\bm{c}_1}_{\eqqcolon\bm{w}_1\ne\emptyset}\cdots\underbrace{\bm{a}_s\bm{b}_s\bm{c}_s}_{\eqqcolon\bm{w}_s\ne\emptyset}}\sum_{M>m_1>\cdots>m_s>0}&\left(\delta(\bm{w}_1)+\zeta^\ast(\bm{a}_1\rfloor)T(m_1\tau)(\bm{b}_1)\zeta^\ast_-(\lfloor \bm{c}_1)\right)\\
        &\cdots\left(\delta(\bm{w}_s)+\zeta^\ast(\bm{a}_s\rfloor)T(m_s\tau)(\bm{b}_s)\zeta^\ast_-(\lfloor\bm{c}_s)\right).
    \end{align}
\end{lem}
\begin{dfn}[\cite{Ba2}]
    We define the mould of the harmonic regularized multiple Eisenstein series by $G^\ast(\tau)\coloneqq \MEI^\ast(\tau)\times\zeta^\ast$, i.e. for $\bm{k}=(k_1,\dots,k_r)\in(\ZZ_{>0})^r$, $G^\ast_{\bm{k}}(\tau)=\langle G^\ast(\tau)\mid\bm{k}\rangle$ is expressed as follows.
    \begin{align}
        G^\ast_{k_1,\dots,k_r}(\tau)=\sum_{i=0}^r\MEI^\ast_{k_1,\dots,k_i}(\tau)\zeta^\ast(k_{i+1},\dots,k_r),
    \end{align}
    where $\MEI^\ast_{\bm{k}}(\tau)\coloneqq\langle \MEI^\ast(\tau)\mid\bm{k}\rangle$.
\end{dfn}
\begin{prop}[\cite{Ba2}]
    \begin{enumerate}
        \item For $\bm{k}\in\mathbb{I}^{\ge2}$, we have $G^\ast_{\bm{k}}(\tau)=G_{\bm{k}}(\tau)$.
        \item $G^\ast(\tau)$ satisfy the harmonic product $\ast$.
    \end{enumerate}
\end{prop}

\subsection{Proof of \Cref{thm:comparison of regularization}}

In this subsection, we give the proof of \Cref{thm:comparison of regularization}. For a positive integer $M>0$, we define the $v$-mould $g^M(\tau)$ by
    \begin{align}
        g^M(\tau)(v_1,\dots,v_r)\coloneqq\sum_{M>m_1>\cdots>m_r>0}T(m_1\tau)(v_1)\cdots T(m_r\tau)(v_r).
    \end{align}
The coefficients $g^M_{\bm{k}}(\tau)\coloneqq\langle g^M(\tau)\mid\bm{k}\rangle$ are given by
    \begin{align}
        g^M_{k_1,\dots,k_r}(\tau)=\sum_{M>m_1>\cdots>m_r>0}\Psi_{k_1}(m_1\tau)\cdots\Psi_{k_r}(m_r\tau).
    \end{align}
for $k_1,\dots,k_r\ge1$.
\begin{lem}[\cite{Ba2}]
    When $k_1\ge2$, the limit $\widetilde{g}_{k_1,\dots,k_r}(\tau)\coloneqq\lim_{M\to\infty}g^M_{k_1,\dots,k_r}(\tau)$ exists and defines a holomorphic function on $\HH$. Moreover, $\widetilde{g}_{\bm{k}}(\tau)$ coincides with $g_{\bm{k}}(\tau)$ when $\bm{k}\in\mathbb{I}^{\ge2}$.
\end{lem}
Let $\mathbb{I}^{\ge2,\mathrm{alm}}$ be the set of all indices $\bm{k}=(k_1,\dots,k_r)\in(\ZZ_{>0})^r$ with at most one entry equal to $1$, located at some position $i\in\{2,\dots,r\}$, and all other entries $\ge2$, i.e.
\begin{align}
    \mathbb{I}^{\ge2,\mathrm{alm}}\coloneqq\mathbb{I}^{\ge2}\sqcup\{(k_1,\dots,k_r)\in(\ZZ_{>0})^r\mid r\ge2, k_1,\dots,k_{i-1},k_{i+1},\dots,k_r\ge2,k_i=1\;(2\le i\le r)\}.
\end{align}
\begin{lem}\label{lem:hast}
    For $\bm{k}\in\mathbb{I}^{\ge2,\mathrm{alm}}$, we have
    \begin{align}
        \MEI^\ast_{\bm{k}}(\tau)=\lim_{M\to\infty}\left\langle\gilat_{\zeta^\ast}(g^M(\tau))\mid\bm{k}\right\rangle.
    \end{align}
\end{lem}

\begin{proof}
    Since $\co{\MEI}_M(\tau)$ is a constant mould, we have
    \begin{align}
        \langle \MEI^\ast(\tau)\mid\bm{k}\rangle=\lim_{M\to\infty}\langle \mathsf{do}{\MEI}_M(\tau)\mid\bm{k}\rangle
    \end{align}
    for $\bm{k}\in\mathbb{I}^{\ge2,\mathrm{alm}}$. Note that the terms involving the constant mould $\delta$ in \Cref{lem:doh} do not contribute to the coefficient at $\bm{k}$: the mould $\delta$ is supported in even depths $\ge2$ and contributes in each block only to indices whose entries are all equal to $1$, while $\bm{k}\in\mathbb{I}^{\ge2,\mathrm{alm}}$ contains at most one $1$. Therefore, by \Cref{lem:doh}, we have
    \begin{align}
        \langle \mathsf{do}{\MEI}_M(\tau)\mid\bm{k}\rangle=\bigg\langle\summ{\bm{v}=\bm{a}_1\bm{b}_1\bm{c}_1\cdots\bm{a}_s\bm{b}_s\bm{c}_s\\\bm{b}_i\ne\emptyset}\sum_{M>m_1>\cdots>m_s>0}&\zeta^\ast(\bm{a}_1\rfloor)T(m_1\tau)(\bm{b}_1)\zeta^\ast_-(\lfloor \bm{c}_1)\\
        &\cdots\zeta^\ast(\bm{a}_s\rfloor)T(m_s\tau)(\bm{b}_s)\zeta^\ast_-(\lfloor \bm{c}_s)\;\bigg|\;\bm{k}\bigg\rangle.
    \end{align}
    Since $T(z)(v_1,\dots,v_r)$ vanishes if $r\ge2$, the sum inside the right-hand side can be written as follows.
    \begin{align}
        &\summ{\bm{v}=\bm{a}_1\bm{b}_1\bm{c}_1\cdots\bm{a}_s\bm{b}_s\bm{c}_s\\\bm{b}_i,\bm{c}_i\bm{a}_{i+1}\ne\emptyset}g^M(\tau)(\bm{b}_1\cdots \bm{b}_s)\zeta^\ast(\bm{a}_1\rfloor)\cdots\zeta^\ast(\bm{a}_s\rfloor)\zeta^\ast_-(\lfloor\bm{c}_1)\cdots\zeta^\ast_-(\lfloor\bm{c}_s)
        =\left(\gilat_{\zeta^\ast}(g^M(\tau))\right)(\bm{v}).
    \end{align}
    This completes the proof.
\end{proof}
\begin{lem}\label{lem:gcomp}
For $\bm{k}=(k_1,\dots,k_{l-1},1,k_{l+1},\dots,k_r)\in\mathbb{I}^{\ge2,\mathrm{alm}}$ with $2\le l\le r$ and $k_i\ge2$ for all $i\ne l$, we have
    \begin{align}
        g^\shuffle_{\bm{k}}(\tau)=\widetilde{g}_{\bm{k}}(\tau).
    \end{align}
\end{lem}

\begin{proof}
    By definition of $g^\shuffle$ (\Cref{dfn:shufflereg}), we have
    \begin{align}
        g^\shuffle_{\bm{k}}(\tau)&=g_{\bm{k}}(\tau)+\left\langle-2\pi i\cdot H(\tau)\binom{\overbrace{1,\dots,1}^{r-l},\color{red}2\color{black},\overbrace{1,\dots,1}^{l-2}}{v_r,\dots,v_{l+1}-v_{l+2},\color{red}v_{l-1}-v_{l+1}\color{black},v_{l-2}-v_{l-1},\dots,v_1-v_2}\;\bigg|\;\bm{k}\right\rangle.
    \end{align}
    By \eqref{eq:H}, one can compute
    \begin{align}
        &H(\tau)\binom{\overbrace{1,\dots,1}^{r-l},2,\overbrace{1,\dots,1}^{l-2}}{v_r,\dots,v_{l+1}-v_{l+2},v_{l-1}-v_{l+1},v_{l-2}-v_{l-1},\dots,v_1-v_2}\\
        &=\frac{(-2\pi i)^r}{2}\!\!\!\!\sum_{d_r>\cdots>d_{l-1}>d_{l+1}>\cdots>d_1>0}\left(\frac{q^{d_{l-1}}}{1-q^{d_{l-1}}}\right)^2\prod_{\substack{j=1\\j\ne l-1,l}}^r\frac{q^{d_j}}{1-q^{d_j}}e^{-2\pi i(d_{l+1}-d_{l-1})v_{l+1}}\prod_{\substack{j=1\\j\ne l,l+1}}^{r}e^{-2\pi i(d_j-d_{j-1})v_j}\\
        &=\frac{1}{2}\summ{k_1,\dots,k_{l-1},\\k_{l+1},\dots,k_r>0}\left\{\prod_{\substack{j=1\\j\ne l}}^r\frac{(-2\pi i)^{k_j}}{(k_j-1)!}\summ{c_1>\cdots>c_{l-1}>c_{l+1}>\cdots>c_r>0\\d_1,\dots,d_{l-1},d_{l+1},\dots,d_r>0}
        (c_{l-1}-c_{l+1}-1)\prod_{\substack{j=1\\j\ne l}}^rd_j^{k_j-1}q^{c_jd_j}\right\}v_1^{k_1-1}\cdots v_r^{k_r-1}.
    \end{align}
    Therefore we have
    \begin{align}
        g^\shuffle_{\bm{k}}(\tau)&=g_{\bm{k}}(\tau)+(-2\pi i)\sum_{c_1>\cdots >c_r>0}\Psi_{k_1}(c_1\tau)\cdots\Psi_{k_{l-1}}(c_{l-1}\tau)\cdot\frac{1}{2}\cdot\Psi_{k_{l+1}}(c_{l+1}\tau)\cdots\Psi_{k_r}(c_r\tau)\\
        &=\sum_{c_1>\cdots>c_r>0}\Psi_{k_1}(c_1\tau)\cdots\Psi_{k_{l-1}}(c_{l-1}\tau)\Psi_1(c_l\tau)\Psi_{k_{l+1}}(c_{l+1}\tau)\cdots\Psi_{k_r}(c_r\tau)\\
        &=\widetilde{g}_{\bm{k}}(\tau).
    \end{align}
\end{proof}
\begin{lem}\label{lem:zetacomp}
    For $(1,\bm{k})$ with $\bm{k}$ admissible, we have
    \begin{align}
        \zeta^\shuffle(1,\bm{k})=\zeta^\ast(1,\bm{k}).
    \end{align}
\end{lem}
\begin{proof}
    Let $Z^\shuffle_{\bm{k}}(T)$ and $Z^\ast_{\bm{k}}(T)$ be the regularized polynomials of multiple zeta values constructed by Ihara--Kaneko--Zagier in \cite{IKZ}. Note that the degree of $Z^\shuffle_{1,\bm{k}}(T)$ and $Z^\ast_{1,\bm{k}}(T)$ are $1$ for any admissible index $\bm{k}$. In the \emph{regularization theorem};
    \begin{align}
        Z^\shuffle_{\bm{k}}(T)=\rho\left(Z^\ast_{\bm{k}}(T)\right),
    \end{align}
    the map $\rho$ is determined by $\rho(1)=1,\;\rho(T)=T,\dots$. Therefore $Z^\shuffle_{1,\bm{k}}(T)$ coincides with $Z^\ast_{1,\bm{k}}(T)$, and we have
    \begin{align}
        \zeta^\shuffle(1,\bm{k})=Z^\shuffle_{1,\bm{k}}(0)=Z^\ast_{1,\bm{k}}(0)=\zeta^\ast(1,\bm{k})
    \end{align}
    for any admissible index $\bm{k}$.
\end{proof}
\newtheorem*{regcomp}{\Cref{thm:comparison of regularization} (i)}
\begin{regcomp}
    For $\bm{k}=(k_1,\dots,k_r)\in\mathbb{I}^{\ge2,\mathrm{alm}}$, we have
    \begin{align}
        G^\shuffle_{\bm{k}}(\tau)=G^\ast_{\bm{k}}(\tau).
    \end{align}
\end{regcomp}

\begin{proof}
    By \Cref{lem:zetacomp}, it suffices to show that
    \begin{align}
        \MEI^\ast_{\bm{k}}(\tau)=\left\langle\gilat_{\zeta^\shuffle}(g^\shuffle(\tau))\relmid\bm{k}\right\rangle
    \end{align}
    for $\bm{k}=(k_1,\dots,k_r)\in\mathbb{I}^{\ge2,\mathrm{alm}}$.
    By \Cref{lem:hast} and \Cref{lem:gila}, $\MEI^\ast_{k_1,\dots,k_r}(\tau)$ equals
    \begin{align}\label{eq:hast}
        &\qquad \lim_{M\to\infty}\summ{0=t_1<q_1\le t_2<q_2\le\cdots\le t_s<q_s\le t_{s+1}=r\\(1\le s\le r)}\summ{n_{t_1+1}+\cdots+n_{t_2}=k_{t_1;t_2}\\\vdots\\n_{t_s+1}+\cdots+n_{t_{s+1}}=k_{t_s;t_{s+1}}}g^M_{n_{q_1},\dots,n_{q_s}}(\tau)\\
        &\times\prod_{j=1}^s\bigg\{(-1)^{k_{t_j;t_{j+1}}+k_{q_j}+n_{t_j;q_j-1}}\Bigg(\prod_{\substack{p=t_j+1\\p\ne q_j}}^{t_{j+1}}\binom{n_p-1}{k_p-1}\Bigg)\zeta^\ast(n_{t_j+1},\dots,n_{q_j-1})\zeta^\ast(n_{t_{j+1}},\dots,n_{q_j+1})\bigg\}.
    \end{align}
    Note that all the indices of $\zeta^\ast$ in \eqref{eq:hast} contain at most one $1$ since the binomial coefficients vanish when $n_p<k_p$ for $p\in\{t_j+1,\dots,q_{j-1},q_{j+1},\dots,t_{j+1}\}$. Therefore we can replace $\zeta^\ast$ with $\zeta^\shuffle$ in \eqref{eq:hast}.
    Now, by \Cref{lem:gcomp}, it remains to show that the only term arising for $n_{q_j}=1$ is $\widetilde{g}_{\bm{k}}(\tau)$. We can write the term in \eqref{eq:hast} arising for $n_{q_j}=1$ for some $j$ and $1\le s<r$ as follows.
    \begin{align}
        &\lim_{M\to\infty}\summ{t_2,\dots,t_s\\q_1,\dots,q_{j-1},q_{j+1},\dots,q_s\\n_1,\dots,n_{t_j},n_{t_{j+1}+1},\dots,n_r}(-1)^{-}\prod\binom{-}{-}\zeta^\shuffle(-)\zeta^\shuffle(-)g^M_{n_{q_1},\dots,n_{q_{j-1}},1,n_{q_{j+1}},\dots,n_{q_s}}(\tau)\\
        &\times\left(\sum_{q_j=t_j+1}^{t_{j+1}}\summ{n_{t_j+1},\dots,n_{t_{j+1}}\ge0\\n_{t_j;t_{j+1}}=k_{t_j;t_{j+1}}\\n_{q_j}=1}(-1)^{n_{t_j;q_j-1}+k_{q_j}}\prod_{\substack{p=t_j+1\\p\ne q_j}}^{t_{j+1}}\binom{n_p-1}{k_p-1}\zeta^\shuffle(n_{t_j+1},\dots,n_{q_j-1})\zeta^\shuffle(n_{t_{j+1}},\dots,n_{q_j+1})\right).
    \end{align}
    By the shuffle antipode relations for the shuffle regularized multiple zeta values, we know that these terms vanish.
\end{proof}

\section{Realization of the formal multiple Eisenstein series}\label{sec:bi MES}

In this section, we define the \emph{bi-multiple Eisenstein series}, which is a realization of the formal multiple Eisenstein series developed in \cite{BIM}.

\subsection{Formal multiple Eisenstein series}\label{subsec:formal MES}

In this subsection, we briefly review the notion of formal multiple Eisenstein series following \cite{BIM}. For this let $$\A\coloneqq\left\{\ai{k}{d}\mid k\geq1,d\geq0\right\}$$ and denote by $\widehat{\ha}^1\coloneqq\QA$ the non-commutative algebra generated by the elements in $\A$, which we refer to as letters. We use the following notation to write words in $\hha^1$:
\begin{align}
    \ai{k_1,\dots,k_r}{d_1,\dots,d_r}\coloneqq\ai{k_1}{d_1}\cdots\ai{k_r}{d_r}
\end{align}
for $k_1,\dots,k_r\geq1$ and $d_1,\dots,d_r\geq0$. Define the harmonic product $\ast$ on $\hha^1$ recursively by
\begin{align}
    \ai{k_1}{d_1}u\ast\ai{k_2}{d_2}v=\ai{k_1}{d_1}\left(u\ast\ai{k_2}{d_2}v\right)+\ai{k_2}{d_2}\left(\ai{k_1}{d_1}u\ast v\right)+\ai{k_1+k_2}{d_1+d_2}\left(u\ast v\right)
\end{align}
for any letters $\ai{k_1}{d_1},\ai{k_2}{d_2}$ and words $u,v\in\hha^1$. This product endows $\hha^1$ with the structure of an associative commutative algebra, and we have a natural embedding of $\QQ$-algebras via the map
\begin{align}
    \ha^1_\ast&\longrightarrow\hha^1_\ast\\
    z_{k_1}\cdots z_{k_r}&\longmapsto\ai{k_1,\dots,k_r}{0,\dots,0}.
\end{align}
This setup fits naturally with the notion of bimoulds in the context of generating functions. As seen in \Cref{rem:identification},
for a given commutative $\QQ$-algebra $R$, linear map $M:\hha^1\rightarrow R$ and bimould $M=(M_r)_{r\geq0}$ over $R$ can be identified with each other via the following:
\begin{align}
    M_r\binom{u_1,\dots,u_r}{v_1,\dots,v_r}=\summ{k_1,\dots,k_r\geq1\\d_1,\dots,d_r\geq0}M\left(\ai{k_1,\dots,k_r}{d_1,\dots,d_r}\right)\frac{u_1^{d_1}}{d_1!}\cdots\frac{u_r^{d_r}}{d_r!}v_1^{k_1-1}\cdots v_r^{k_r-1}.
\end{align}
Under this identification, we say a bimould $M$ is \emph{symmetril} if the corresponding linear map $M:\hha^1\rightarrow R$ is algebra homomorphism in terms of the harmonic product $\ast$. If $M$ is symmetril, then in depth $2$ we have
\begin{align}
    M\binom{u_1}{v_1}M\binom{u_2}{v_2}=M\binom{u_1,u_2}{v_1,v_2}+M\binom{u_2,u_1}{v_2,v_1}+\frac{M\binom{u_1+u_2}{v_1}-M\binom{u_1+u_2}{v_2}}{v_1-v_2}.
\end{align}
\begin{rem}\label{rem:symmetril}
    A similar argument to \Cref{rem:gila and Goncharov} also holds for $\hha^1_\ast$. It is known that $\hha^1_\ast$ has the Hopf algebra structure with the deconcatenation coproduct $\Delta_\mathrm{dec}$ (see \cite{Ho}). For $\QQ$-algebra homomorphisms $A,B:\hha^1_\ast\rightarrow R$, the convolution of $A$ and $B$ coincides with the product of bimoulds, i.e. the linear map $A\times B:\hha^1\rightarrow R$ corresponding to the bimould $A\times B\in\BIMU$ is given by $A\times B=\mu\circ(A\otimes B)\circ\Delta_{\mathrm{dec}}$. And therefore, the product of bimoulds $A\times B$ is symmetril if $A$ and $B$ are both symmetril.
\end{rem}
Now we define the \emph{swap operator} $\sigma$ via the generating function. Let $\fA=(\fA_r)_{r\geq0}$ be a bimould over $\hha^1$ defined by
\begin{align}
    \fA_r\binom{u_1,\dots,u_r}{v_1,\dots,v_r}\coloneqq\summ{k_1,\dots,k_r\geq1\\d_1,\dots,d_r\geq0}\ai{k_1,\dots,k_r}{d_1,\dots,d_r}\frac{u_1^{d_1}}{d_1!}\cdots\frac{u_r^{d_r}}{d_r!}v_1^{k_1-1}\cdots v_r^{k_r-1}
\end{align}
for $r\geq1$ and $\fA_0\coloneqq1$. Note that for a given map $\rho:\BIMU(\hha^1)\rightarrow\BIMU(\hha^1)$, we can obtain a map $\rho^\prime:\hha^1\rightarrow\hha^1$ by comparing coefficients, i.e. $\rho^\prime\left(\ai{k_1,\dots,k_r}{d_1,\dots,d_r}\right)$ is defined as the coefficient of $\frac{u_1^{d_1}}{d_1!}\cdots\frac{u_r^{d_r}}{d_r!}v_1^{k_1-1}\cdots v_r^{k_r-1}$ in $\rho(\fA)\binom{u_1,\dots,u_r}{v_1,\dots,v_r}$.
\begin{dfn}\label{def:swap}
    We define the linear map $\sigma:\hha^1\rightarrow\hha^1$ given by $\swap:\BIMU\rightarrow\BIMU$, i.e.
\begin{align}
    \summ{k_1,\dots,k_r\geq1\\d_1,\dots,d_r\geq0}\sigma\left(\ai{k_1,\dots,k_r}{d_1,\dots,d_r}\right)\frac{u_1^{d_1}}{d_1!}\cdots\frac{u_r^{d_r}}{d_r!}v_1^{k_1-1}\cdots v_r^{k_r-1}=\fA_r\binom{v_r,v_{r-1}-v_r,\dots,v_1-v_2}{u_1+\cdots+u_r,u_1+\cdots+u_{r-1},\dots,u_1}.
\end{align}
\end{dfn}
The formal multiple Eisenstein series are objects whose generating series are symmetril and $\swap$-invariant.
\begin{dfn}\label{def:formal MES}
    We define the space of \emph{formal multiple Eisenstein series} by
    \begin{align}
        \bmes^f\coloneqq\quotient{\hha^1_\ast}{\I},
    \end{align}
    where $\I$ is the ideal in $\hha^1_\ast$ generated by $\sigma(w)-w$ for all $w\in\hha^1$. Denote the class of word $\ai{k_1,\dots,k_r}{d_1,\dots,d_r}$ by $G^f\binom{k_1,\dots,k_r}{d_1,\dots,d_r}$. For the class of word in $\ha^1$, we write
    \begin{align}
        G^f(k_1,\dots,k_r)\coloneqq G^f\binom{k_1,\dots,k_r}{0,\dots,0}.
    \end{align}
    We also define the subspaces $\bmes^{f,0}$ and $\mes^f$ by
    \begin{align}
         \bmes^{f,0}&\coloneqq\Span_\QQ\{G^f(k_1,\dots,k_r)\in\bmes^f\mid r\ge0, k_1,\dots,k_r\ge1\},\\
         \mes^f&\coloneqq\Span_\QQ\{G^f(k_1,\dots,k_r)\in\bmes^f\mid r\ge0,k_1,\dots,k_r\ge2\}.
    \end{align}
\end{dfn}
The main result in \cite{BIM} constructs $\sigma$-equivariant derivations $W,D,\delta$ such that $\bmes^f$ has an $\sltwo$-algebra structure with these derivations. The $\QQ$-linear maps $W,D:\hha^1\rightarrow\hha^1$ are defined by
\begin{align}
    W(w)&\coloneqq(k_1+\cdots+k_r+d_1+\cdots+d_r)\ai{k_1,\dots,k_r}{d_1,\dots,d_r},\\
    D(w)&\coloneqq\sum_{j=1}^rk_j\ai{k_1,\dots,k_j+1,\dots,k_r}{d_1,\dots,d_j+1,\dots,d_r}
\end{align}
for a word $w=\ai{k_1,\dots,k_r}{d_1,\dots,d_r}\in\hha^1$. In this paper, we do not give an explicit definition of the map $\delta:\hha^1\rightarrow\hha^1$ (see \cite{BIM}). Instead, we describe the image of $\delta$ restricted to $\ha^1$ later \eqref{eq:delta in H1}.
\begin{thm}[\cite{BIM}]
    $W,D,\delta$ are $\sigma$-equivariant derivations on $\hha^1_\ast$. Moreover, these derivations form an $\sltwo$-triple, and thus $\bmes^f$ is an $\sltwo$-algebra.
\end{thm}
It is natural to ask whether the subalgebras $\bmes^{f,0}$ and $\mes^f$ inherit the $\sltwo$-algebra structure. In \cite{BIM}, the authors gave a positive answer for the former algebra. In \Cref{subsec:derivative}, we give a positive answer for the latter algebra. The maps $\delta,W$ are closed in $\ha^1$ (in particular in $\bmes^{f,0}$), and the image of $\delta$ is given by
\begin{align}\label{eq:delta in H1}
    \delta(w)&=\left(\frac{1}{2}\mathbf{1}_{w=(yx)u}-\frac{1}{2}\mathbf{1}_{w=(xy)u}+\frac{1}{4}\mathbf{1}_{w=(yy)u}\right)u+\frac{1}{2}\left(\sum_{w=u(yyx)u^\prime}-\sum_{w=u(xyy)u^\prime}\right)uyu^\prime
\end{align}
for a word $w=z_{k_1}\cdots z_{k_r}=x^{k_1-1}y\cdots x^{k_r-1}y\in\ha^1$. In \cite{BIM}, the authors showed that $\bmes^{f,0}$ is closed under $D$.
\begin{thm}[\cite{BIM}]\label{thm:sltwo for gf0}
    $\dsht=D$ in $\bmes^{f,0}$. In particular, $\bmes^{f,0}$ is an $\sltwo$-subalgebra of $\bmes^f$.
\end{thm}

By \Cref{thm:sltwo for gf0}, the commutator relation $\lbrack\delta,\dsht\rbrack=W$ holds in $\bmes^{f,0}$. This relation actually holds in $\ha^1$.
\begin{lem}\label{lem:commutator}
    For any $w\in\ha^1$, we have $\lbrack\delta,\dsht\rbrack(w)=W(w)$.
\end{lem}
\begin{proof}
    By direct computation, we have
    \begin{align}
        D(w)-\dsht(w)=\sigma(\sigma(w)\ast\sigma(z_2))-w\ast z_2
    \end{align}
    for any $w\in\ha^1$. Thus, we have
    \begin{align}
        \lbrack\delta,\dsht\rbrack(w)&=\lbrack\delta,D\rbrack(w)+\delta(w\ast z_2)-\delta(w)\ast z_2-\sigma\left(\delta(\sigma(w)\ast\sigma(z_2))\right)+\sigma\left(\delta(\sigma(w))\ast\sigma(z_2)\right).
    \end{align}
    Since $\delta$ is a $\sigma$-equivariant derivation for $\hha^1_\ast$, we have
    \begin{align}
        \delta(w\ast z_2)-\delta(w)\ast z_2=-\frac{1}{2}w,\\
        \sigma\left(\delta(\sigma(w)\ast\sigma(z_2))\right)-\sigma\left(\delta(\sigma(w))\ast\sigma(z_2)\right)&=\frac{1}{2}w.
    \end{align}
    Therefore we have $\lbrack\delta,\dsht\rbrack(w)=W(w)$.
\end{proof}
\begin{lem}\label{lem:deltaR}
    For $u,v\in\ha^1$, we have
    \begin{align}
        \delta(R(u,v))=R(\delta(u),v)+R(u,\delta(v)).
    \end{align}
\end{lem}

\begin{proof}
    By \Cref{lem:commutator}, we have
    \begin{align}
        \delta(R(u,v))&=\delta(\dsht(u\ast v)-\dsht(u)\ast v-u\ast \dsht(v))\\
        &=W(u\ast v)+\dsht(\delta(u\ast v))-\delta(\dsht(u)\ast v)-\delta(u\ast \dsht(v)).\label{eq:deltaR}
    \end{align}
    Since $\delta$ is a derivation on $\ha^1_\ast$, we have
    \begin{align}
        \dsht(\delta(u\ast v))=\dsht(\delta(u)\ast v)+\dsht(u\ast\delta(v)),\\
        \delta(\dsht(u)\ast v)=\dsht(\delta(u))\ast v+\dsht(u)\ast\delta(v),\\
        \delta(u\ast \dsht(v))=\delta(u)\ast \dsht(v)+u\ast\delta(\dsht(v)).
    \end{align}
    Again, by \Cref{lem:commutator} and \eqref{eq:deltaR}, we have
    \begin{align}
        \delta(R(u,v))&=W(u\ast v)+\dsht(\delta(u)\ast v)+\dsht(u\ast\delta(v))-W(u)\ast v-\dsht(\delta(u))\ast v\\
        &\quad-\dsht(u)\ast\delta(v)-\delta(u)\ast \dsht(v)-u\ast W(v)-u\ast \dsht(\delta(v))\\
        &=W(u\ast v)-W(u)\ast v-u\ast W(v)+R(\delta(u),v)+R(u,\delta(v)).
    \end{align}
    Since $W$ is a derivation on $\ha^1_\ast$, we have the conclusion.
\end{proof}

In \cite{BIM}, the authors conjectured that the space $\mes^f$ is also an $\sltwo$-subalgebra and the question was whether $\mes^f$ is closed under the derivation $D$. In \Cref{sec:der}, we provide an answer to this question. The following lemma is required in order to prove that $\mes^f$ is closed under $D$.
\begin{lem}\label{lem:shuffle with z_1}
    For $\bm{k}=(k_1,\dots,k_r)\in(\ZZ_{>0})^r$, we have
    \begin{align}
    \sigma(\sigma(z_{\bm{k}})\ast\sigma(z_1))=z_{\bm{k}}\shuffle z_1+\ai{k_1,k_2,\dots,k_r}{1,0,\dots,0},
    \end{align}
    where $z_{\bm{k}}=z_{k_1}\cdots z_{k_r}$.
\end{lem}
\begin{proof}
     First, we can extend the harmonic product on $\hha^1$ to $\varinjlim_{r\geq0}\hha^1\llbracket u_1,v_1,\dots,u_r,v_r\rrbracket$, and the swap operator $\sigma$ can also be extended to $\varinjlim_{r\geq0}\hha^1\llbracket u_1,v_1,\dots,u_r,v_r\rrbracket$. By direct computation, we have
     \begin{align}
         \sigma\left(\fA_1\binom{u}{v}\right)\ast\sigma\left(\fA_r\binom{u_1,\dots,u_r}{v_1,\dots,v_r}\right)
         &=\sum_{j=1}^{r+1}\fA_{r+1}\binom{v_r,\dots,v_{j}-v_{j+1},v,v_{j-1}-v_{j},\dots,v_1-v_2}{u_{1;r},\dots,u_{1;j},u,u_{1;j-1},\dots,u_1}\\
         &+\sum_{j=1}^r\frac{1}{u-u_{1;j}}\bigg(\fA_r\binom{v_r,\dots,v_j-v_{j+1}+v,\dots,v_1-v_2}{u_{1;r},\dots,u,\dots,u_1}\\
         &-\fA_r\binom{v_r,\dots,v_j-v_{j+1}+v,\dots,v_1-v_2}{u_{1;r},\dots,u_{1;j},\dots,u_1}\bigg),
     \end{align}
     where $u_{i;j}=u_i+\cdots+u_j$ for $i\leq j$. By applying $\sigma$, substituting $u_1=\cdots=u_r=0$ and then taking $u\to0$, we have
     \begin{align}
         \sigma\left(\sigma\left(\fA_1\binom{0}{v}\right)\ast\sigma\left(\fA_r\binom{0,\dots,0}{v_1,\dots,v_r}\right)\right)=\sum_{j=1}^{r+1}\fA_{r+1}\binom{0,\dots,0}{v_1+v,\dots,v_j+v,v_j,\dots,v_r}\\
         +\sum_{j=1}^r\left(\frac{\partial}{\partial u_j}-\frac{\partial}{\partial u_{j+1}}\right)\fA_r\binom{u_1,\dots,u_r}{v_1+v,\dots,v_j+v,v_{j+1},\dots,v_r}\bigg|_{u_1=\cdots=u_r=0}.
     \end{align}
     By comparing the coefficients of $v_1^{k_1-1}\cdots v_r^{k_r-1}$ in both sides, we have the conclusion.
\end{proof}

\subsection{Bi-multiple Eisenstein series}

In this subsection, we construct the bi-multiple Eisenstein series as the image of the $\sigma$-invariant algebra homomorphism $G:\hha^1_\ast\rightarrow\mathcal{O}(\HH)$. The key idea of the construction is based on \cite{BB}, where \emph{combinatorial (bi-)multiple Eisenstein series} were introduced. The combinatorial multiple Eisenstein series are defined as the image of a realization map from $\bmes^f$ to $\QQ\llbracket q\rrbracket$, and the construction of such a realization map is non-canonical; more precisely, the realization depends on a choice of a rational solution of the extended double shuffle equations (see \cite{BB}). However, the construction of the bi-multiple Eisenstein series is canonical.

First, we define the bi-multiple zeta values, introduced in \cite{BI}. Define the $u$-mould $\zeta^\natural$ and the constant mould $\delta_T$ ($T$ is an indeterminate) by $\zeta^\natural\coloneqq\swap(\zeta^\shuffle),\delta_T\coloneqq\left(\frac{T^r}{r!}\right)_{r\geq0}$. Specifically, we have
\begin{align}
    \zeta^\natural(u_1,\dots,u_r)&=\zeta^\shuffle(u_1+\cdots+u_r,\dots,u_1)\\
    &=\sum_{k_1,\dots,k_r\geq1}\zeta^\shuffle(k_1,\dots,k_r)(u_1+\cdots+u_r)^{k_1-1}\cdots u_1^{k_r-1}.
\end{align}
Note that these moulds are symmetril. Indeed, it is known that the $u$-mould $\zeta^\natural$ is symmetral (cf. \cite{IKZ}), and the symmetrality is equivalent to the symmetrility in the case of $u$-moulds (cf. \cite{BB}).
\begin{dfn}[\cite{BI}]
    We define the \emph{bi-multiple zeta values} $\bMZV:\hha^1\rightarrow\mz$ by $\bMZV\coloneqq\zeta^\natural\times\zeta^\ast$, i.e.
    \begin{align}
        \bMZV\binom{u_1,\dots,u_r}{v_1,\dots,v_r}=\sum_{j=0}^r\zeta^\natural(u_1,\dots,u_j)\zeta^\ast(v_{j+1},\dots,v_r).
    \end{align}
\end{dfn}
Note that, by \Cref{rem:symmetril}, bi-multiple zeta values satisfy the harmonic product, i.e. $\bMZV:\hha^1_\ast\to\mz$ is an algebra homomorphism (or equivalently, $\bMZV$ is symmetril). We also define the symmetril bimould $\widetilde{\bMZV}$ (or equivalently an algebra homomorphism $\widetilde{\bMZV}:\hha^1_\ast\rightarrow\mz\lbrack\pi i\rbrack$) by $\widetilde{\bMZV}\coloneqq\nega(\zeta^\natural)\times\pari(\delta_{\pi i})\times\zeta^\ast$.
\begin{lem}
    We have $\swap(\widetilde{\bMZV})=\nega(\widetilde{\bMZV})$.
\end{lem}
\begin{proof}
    It is known that $\swap(\bMZV)=\bMZV$ (\cite{BI}). Therefore, we have
    \begin{align}
        \swap(\widetilde{\bMZV})\binom{u_1,\dots,u_r}{v_1,\dots,v_r}&=(-1)^r\sum_{j=0}^r\frac{(-\pi i)^j}{j!}\bMZV\binom{v_r,\dots,v_{j+1}-v_{j+2}}{-u_1-\cdots-u_{r-j},\dots,-u_1}\\
        &=(-1)^r\sum_{j=0}^r\frac{(-\pi i)^j}{j!}\bMZV\binom{-u_1,\dots,-u_{r-j}}{v_{j+1},\dots,v_r}\\
        &=\widetilde{\bMZV}\binom{-u_1,\dots,-u_r}{-v_1,\dots,-v_r}.
    \end{align}
\end{proof}
To construct the bi-multiple Eisenstein series, we consider the modified version of the regularized multitangent function. We define the symmetril $v$-mould $\Psi^+(z)$ by $\Psi^+(z)\coloneqq\Psi^\ast(z)\times\delta_{\pi i}$, specifically, $\Psi^+_{k_1,\dots,k_r}(z)$ is given by
\begin{align}\label{eq:MTF plus}
    \Psi^+_{\bm{k},\{1\}^n}(z)=\summ{j,l\geq0\\j+l=n}\Psi^\ast_{\bm{k},\{1\}^j}(z)\frac{(\pi i)^l}{l!}
\end{align}
for $\bm{k}=(k_1,\dots,k_r)\in(\ZZ_{>0})^r$ ($r\geq0$) with $k_r\geq2$ and $n\geq0$. The two regularized multitangent functions $\Psi^\ast(z)$ and $\Psi^+(z)$ satisfy the same differential equation.
\begin{lem}[\cite{Bo}]\label{lem:der MTF}
    For $k_1,\dots,k_r\geq1$, we have for $\bullet\in\{\ast,+\}$
    \begin{align}
        \frac{d}{dz}\Psi^\bullet_{k_1,\dots,k_r}(z)=-\sum_{j=1}^rk_j\Psi^\bullet_{k_1,\dots,k_j+1,\dots,k_r}(z).
    \end{align}
\end{lem}
\begin{proof}
    The case for $\bullet=\ast$ is already proven by \cite{Bo}. It is enough to show the case for $\Psi^+_{\bm{k},\{1\}^n}(z)$ with $\bm{k}=(k_1,\dots,k_r)$, $k_r\geq2$ and $n\geq1$, which can be obtained by using the differential equation for $\Psi^\ast(z)$ and a direct calculation.
\end{proof}
Using \Cref{thm:reduction}, we can obtain the depth one reduction formula for $\Psi^+(z)$.
\begin{lem}
    The following factorization of mould holds
    \begin{align}\label{eq:reduction plus}
        \Psi^+(z)(\bm{v})=\sum_{\bm{v}=\bm{a}\bm{b}\bm{c}\bm{d}}\zeta^\ast(\bm{a}\rfloor)T^+(z)(\bm{b})\zeta^\ast_-(\lfloor\bm{c})\delta_{\pi i}(\bm{d}),
    \end{align}
    where the $v$-mould $T^+(z)$ is defined by 
    \begin{align}
        T^+(z)(v_1,\dots,v_r)\coloneqq\begin{cases}\sum_{k\geq1}\left(\mathbf{1}_{k=1}\pi i+\Psi_k(z)\right)v_1^{k-1}&r=1\\0&\otherwise\end{cases},
    \end{align}
    and $B_{-}$ denotes $\nega\circ\anti\circ\pari(B)$ for a bimould $B$ as defined in \Cref{subsec:Goncharov}.
\end{lem}
\begin{proof}
    By definition of $\Psi^+(z)$ and \Cref{thm:reduction}, we have
    \begin{align}
        \Psi^+(z)(\bm{v})=\sum_{\bm{v}_1\bm{v}_2=\bm{v}}\delta(\bm{v}_1)\delta_{\pi i}(\bm{v}_2)+\sum_{\bm{v}_1\cdots\bm{v_4}=\bm{v}}\zeta^\ast(\bm{v}_1\rfloor)T(z)(\bm{v}_2)\zeta^\ast_-(\lfloor\bm{v}_3)\delta_{\pi i}(\bm{v}_4).
    \end{align}
    In the case $\bm{k}=(k_1,\dots,k_r,\{1\}^n)$ with $k_r\geq2$ and $n\geq0$, we have
    \begin{align}
        \Psi^+_{\bm{k}}(z)&=\left\langle\sum_{\bm{v}_1\cdots\bm{v_4}=\bm{v}}\zeta^\ast(\bm{v}_1\rfloor)T(z)(\bm{v}_2)\zeta^\ast_-(\lfloor\bm{v}_3)\delta_{\pi i}(\bm{v}_4)\relmid\bm{k}\right\rangle\\
        &=\sum_{j=0}^n\left\langle\sum_{\bm{a}\bm{b}\bm{c}=\bm{v}}\zeta^\ast(\bm{a}\rfloor)T(z)(\bm{b})\zeta^\ast_-(\lfloor\bm{c})\relmid(k_1,\dots,k_r,\{1\}^{n-j})\right\rangle\frac{(\pi i)^j}{j!}.
    \end{align}
    We can replace $T(z)$ by $T^+(z)$ since $\Psi_1(z)$ does not appear in the reduction of $\Psi^\ast_{k_1,\dots,k_r,\{1\}^{n-j}}(z)$. In the case $\bm{k}=\{1\}^n$, we have
    \begin{align}
        \sum_{n\geq0}\Psi^+_{\{1\}^n}(z)t^n&=e^{\pi it}\sum_{n\geq0}\Psi^\ast_{\{1\}^n}(z)t^n.
    \end{align}
    By \eqref{eq:reduction ones}, this can be computed as follows.
    \begin{align}
        e^{\pi it}\sum_{n\geq0}\Psi^\ast_{\{1\}^n}(z)t^n&=e^{\pi it}(\cos\pi t+\sin\pi t\cdot\cot\pi z)\\
        &=1+e^{\pi it}\frac{\sin\pi t}{\pi}\Psi^+_1(z).
    \end{align}
    On the other hand, the coefficients of the right-hand side of \eqref{eq:reduction plus} can be computed as follows.
    \begin{align}
        &\sum_{n\geq0}\left\langle\sum_{\bm{v}=\bm{a}\bm{b}\bm{c}\bm{d}}\zeta^\ast(\bm{a}\rfloor)T^+(z)(\bm{b})\zeta^\ast_-(\lfloor\bm{c})\delta_{\pi i}(\bm{d})\relmid\{1\}^n\right\rangle\\
        &=1+\left(\sum_{i\geq0}\zeta^\ast(\{1\}^i)t^i\right)\Psi^+_1(z)t\left(\sum_{j\geq0}\zeta^\ast(\{1\}^j)(-t)^j\right)\left(\sum_{k\geq0}\frac{(\pi i)^k}{k!}t^k\right)\\
        &=1+e^{\pi it}\frac{\sin\pi t}{\pi}\Psi^+_1(z).
    \end{align}
    Therefore, we have \eqref{eq:reduction plus}.
\end{proof}

\begin{dfn}
    For a positive integer $m>0$, we define the bimoulds $L_m(\tau)$ and $\mathcal{L}_m(\tau)$ by
    \begin{align}
        L_m(\tau)\binom{u_1,\dots,u_r}{v_1,\dots,v_r}&\coloneqq\begin{cases}e^{-2\pi imu_1}T^+(m\tau)(v_1)&r=1\\0&\otherwise\end{cases},\\
        \mathcal{L}_m(\tau)(\bm{w})&\coloneqq\sum_{\bm{w}=\bm{a}\bm{b}\bm{c}}\bMZV(\bm{a}\rfloor)L_m(\tau)(\lceil\bm{b}\rceil)\widetilde{\bMZV}_-(\lfloor\bm{c}).\label{eq:def bMTF}
    \end{align}
\end{dfn}
The following lemma shows that $\mathcal{L}_m(\tau)$ is symmetril.
\begin{lem}\label{lem:reduction bi MTF}
    The following factorization of mould holds
    \begin{align}
        \mathcal{L}_m(\tau)(\bm{w})=e^{-2\pi im|\bm{u}|}(\zeta^\natural\times\Psi^+(m\tau)\times\anti\circ\pari(\zeta^\natural))(\bm{w}),
    \end{align}
    where $|\bm{u}|=u_1+\cdots+u_r$ and $\bm{w}=\binom{u_1,\dots,u_r}{v_1,\dots,v_r}$.
\end{lem}
\begin{proof}
    Since $\widetilde{\bMZV}_-=\nega\circ\anti\circ\pari(\nega(\zeta^\natural)\times\pari(\delta_{\pi i})\times\zeta^\ast)=\zeta^\ast_-\times\delta_{\pi i}\times\pari\circ\anti(\zeta^\natural)$, by \eqref{eq:def bMTF} we have
    \begin{align}
        \mathcal{L}_m(\tau)(\bm{w})
        &=e^{-2\pi im|\bm{u}|}\summ{\binom{\bm{u}_1}{\bm{v}_1}\cdots\binom{\bm{u}_6}{\bm{v}_6}=\bm{w}}\zeta^\natural(\bm{u}_1)\zeta^\ast(\bm{v}_2\rfloor)T^+(m\tau)(\bm{v}_3)\zeta^\ast_-(\lfloor\bm{v}_4)\delta_{\pi i}\binom{\bm{u}_5}{\bm{v}_5}\anti\circ\pari(\zeta^\natural)(\bm{u}_6)\\
        &=e^{-2\pi im|\bm{u}|}(\zeta^\natural\times\Psi^+(m\tau)\times\anti\circ\pari(\zeta^\natural))(\bm{w}).
    \end{align}
\end{proof}

\begin{dfn}
    We define the bimoulds $\bMEI(\tau)$ and $\bMES(\tau)$ by
    \begin{align}
        \bMEI(\tau)(\bm{w})&\coloneqq\lim_{M\to\infty}\summ{1\leq s\leq r\\\bm{w}_1\cdots\bm{w}_s=\bm{w}}\sum_{M>m_1>\cdots>m_s>0}\bMTF{m_1}(\bm{w}_1)\cdots\bMTF{m_s}(\bm{w}_s),\label{eq:def bMEI}\\
        \bMES(\tau)&\coloneqq\bMEI(\tau)\times\bMZV.
    \end{align}
\end{dfn}
We use the same notation $\bMES=\bMES(\tau)$ for the corresponding linear map from $\hha^1$ to $\mathcal{O}(\HH)$, and denote by $\bMES\binom{k_1,\dots,k_r}{d_1,\dots,d_r}(\tau)$ the image of $\ai{k_1,\dots,k_r}{d_1,\dots,d_r}\in\hha^1$ (or equivalently, the coefficient of $\frac{u_1^{d_1}}{d_1!}\cdots\frac{u_r^{d_r}}{d_r!}v_1^{k_1-1}\cdots v_r^{k_r-1}$ in $\bMES(\tau)\binom{u_1,\dots,u_r}{v_1,\dots,v_r}$) and call these the bi-multiple Eisenstein series.
\begin{rem}
    Since the constant term of the Fourier expansion of $\Psi^+(z)$ is zero, the limit of \eqref{eq:def bMEI} exists and thus $\bMEI$ is well-defined.
\end{rem}

\begin{prop}\label{prop:bMES H2}
    $\bMES\mid_{\ha^{\geq2}}=G$, i.e. for any $k_1,\dots,k_r\geq2$, we have
    \begin{align}
        \bMES\binom{k_1,\dots,k_r}{0,\dots,0}(\tau)=G_{k_1,\dots,k_r}(\tau).
    \end{align}
\end{prop}
\begin{proof}
    By \Cref{lem:decom G}, it suffices to show $\langle\bMEI(\tau)\mid\bm{k}\rangle=\MEI_{\bm{k}}(\tau)$ (the right-hand side is defined in \eqref{eq:def MEI}) and $\bMZV\binom{\bm{k}}{\bm{0}}=\zeta(\bm{k})$ for any $\bm{k}\in\mathbb{I}^{\geq2}$. By \Cref{lem:reduction bi MTF}, for any index $\bm{k}=(k_1,\dots,k_r)\in(\ZZ_{>0})^r$, the coefficient $\langle\mathcal{L}_m(\tau)\mid \bm{k}\rangle$ of $v_1^{k_1-1}\cdots v_r^{k_r-1}$ (in $u$-degree $0$) in the bimould $\mathcal{L}_m(\tau)$ coincides with $\Psi^+_{k_1,\dots,k_r}(m\tau)$. And by definition of $\Psi^+$ \eqref{eq:MTF plus}, we have $\langle\mathcal{L}_m(\tau)\mid\bm{k}\rangle=\Psi_{\bm{k}}(m\tau)$ for $\bm{k}\in\mathbb{I}^{\geq2}$, and thus $\langle\bMEI(\tau)\mid\bm{k}\rangle=\MEI_{\bm{k}}(\tau)$. By construction of $\bMZV$, we have $\bMZV\binom{k_1,\dots,k_r}{0,\dots,0}=\zeta(k_1,\dots,k_r)$ for any admissible index $(k_1,\dots,k_r)$. 
\end{proof}

\begin{thm}\label{thm:symmetril bMES}
    The bimould $\bMES$ is symmetril.
\end{thm}
\begin{proof}
    It suffices to prove $\bMEI$ is symmetril. By \Cref{lem:reduction bi MTF}, $\mathcal{L}_m(\tau)$ is symmetril for any $m>0$. By a similar fact to \Cref{lem:f to F}, one can check that $\bMEI(\tau)$ is symmetril.
\end{proof}

\begin{thm}\label{thm:swap bMES}
    The bimould $\bMES$ is $\swap$-invariant, i.e. we have $\bMES(\sigma(w))=\bMES(w)$ for any $w\in\hha^1$.
\end{thm}
\begin{proof}
    For $j\geq0$, we define the bimould $\bMES_j=(\bMES_{j,r})_{r\geq0}$ as follows. In the case $j=0$ we set $\bMES_0=\bMZV$ and in general $\bMES_{j,r}=0$ for $j>r$. If $1\leq j\leq r$ we define
    \begin{align}
        \bMES_j\binom{u_1,\dots,u_r}{v_1,\dots,v_r}\coloneqq\summ{0=r_0<r_1<\cdots<r_j\leq r\\m_1>\cdots>m_j>0}\prod_{i=1}^j\mathcal{L}_{m_i}\binom{u_{r_{i-1}+1},\dots,u_{r_i}}{v_{r_{i-1}+1},\dots,v_{r_i}}\bMZV\binom{u_{r_j+1},\dots,u_r}{v_{r_j+1},\dots,v_r}.
    \end{align}
    This bimould is inspired by the bimould $\mathfrak{G}_j$ defined in \cite{BB}, which is defined to prove the $\swap$-invariance of the combinatorial multiple Eisenstein series. In our setting, we can show that $\swap(\bMES_j)=\bMES_j$ for any $j\geq0$ by the similar manipulation to \cite[Theorem 6.26]{BB}. Furthermore, by definitions of $\bMEI$, $\bMES$ and $\bMES_j$, we have
    \begin{align}
        \bMES\binom{u_1,\dots,u_r}{v_1,\dots,v_r}=\sum_{j=0}^r\bMES_j\binom{u_1,\dots,u_r}{v_1,\dots,v_r}.
    \end{align}
    Therefore, $\bMES$ is $\swap$-invariant.
\end{proof}

Let $\bmes$ be the space spanned by all bi-multiple Eisenstein series. By combining \Cref{thm:symmetril bMES} and \Cref{thm:swap bMES}, we know that the map $\bMES:\hha^1\rightarrow\mathcal{O}(\HH)$ factors through the quotient $\bmes^f$, i.e. we have the well-defined surjective algebra homomorphism $\bMES:\bmes^f\rightarrow\bmes$. This map commutes with the derivation $D$ in $\bmes^f$ and the derivative operator $2\pi i\frac{d}{d\tau}$.

\begin{prop}\label{prop:der bMES}
    For $k_1,\dots,k_r\geq1,d_1,\dots,d_r\geq0$, we have
    \begin{align}
        2\pi i\frac{d}{d\tau}G\binom{k_1,\dots,k_r}{d_1,\dots,d_r}(\tau)=\sum_{j=1}^rk_jG\binom{k_1,\dots,k_j+1,\dots,k_r}{d_1,\dots,d_j+1,\dots,d_r}(\tau).
    \end{align}
\end{prop}
\begin{proof}
    It suffices to show
    \begin{align}\label{eq:bMEI der}
        2\pi i\frac{d}{d\tau}\bMEI(\tau)\binom{u_1,\dots,u_r}{v_1,\dots,v_r}=\sum_{j=1}^r\frac{\partial}{\partial u_j}\frac{\partial}{\partial v_j}\bMEI(\tau)\binom{u_1,\dots,u_r}{v_1,\dots,v_r}.
    \end{align}
    By \Cref{lem:der MTF}, we have
    \begin{align}
        2\pi i\frac{d}{d\tau}\Psi^+(m\tau)(v_1,\dots,v_r)=(-2\pi im)\sum_{j=1}^r\frac{\partial}{\partial v_j}\Psi^+(m\tau)(v_1,\dots,v_r),
    \end{align}
    and by \Cref{lem:reduction bi MTF}, we have
    \begin{align}
        2\pi i\frac{d}{d\tau}\mathcal{L}_m(\tau)\binom{u_1,\dots,u_r}{v_1,\dots,v_r}
        &=\sum_{0\leq i\leq j\leq r}\sum_{l=i+1}^j\frac{\partial}{\partial u_l}\frac{\partial}{\partial v_l}e^{-2\pi im|\bm{u}|}(-1)^{r-j}\zeta^\natural(u_1,\dots,u_i)\\
        &\quad\times\Psi^+(m\tau)(v_{i+1},\dots,v_j)\zeta^\natural(u_r,\dots,u_{j+1})\\
        &=\sum_{l=1}^r\frac{\partial}{\partial u_l}\frac{\partial}{\partial v_l}\mathcal{L}_m(\tau)\binom{u_1,\dots,u_r}{v_1,\dots,v_r}.
    \end{align}
    By definition of $\bMEI$ \eqref{eq:def bMEI}, we have \eqref{eq:bMEI der}.
\end{proof}

\section{Explicit formula for the derivative}\label{sec:der}

In this section, we provide an explicit closed formula for the derivative of $\mes^f$ and $\mes$ by using the Drop1 operator, which was introduced by Hirose, Maesaka, Seki and Watanabe in \cite{HMSW}.

\subsection{Drop1 operator}\label{subsec:Drop1}

In this subsection, we briefly review the results from \cite{HMSW}. Then, we provide some explicit formulas for the Drop1 operator. 

\begin{dfn}[\cite{HMSW}]\label{def:drop1operator}
    Define the Drop1 operator as the $\QQ$-linear map $\mathcal{D}:\ha^0\to\ha^{\ge2}$ defined for a word $w=x^{c_1}y^{c_2}\cdots x^{c_{2s-1}}y^{c_{2s}}\in\ha^0$ by $\mathcal{D}(w)=\mathfrak{D}(\bm{c})=\mathfrak{D}(c_1,\dots,c_{2s})$. Here $\mathfrak{D}$ is recursively given by
    \begin{align}
        \mathfrak{D}(\bm{c})&=\summ{A\subset\lbrack2s\rbrack^1_{\bm{c}}\\A:\even-\odd}\summ{B\subset\lbrack2s\rbrack^{>1}_{\bm{c}}\\\#A+\#B\ge1\\\{2r,2r+1\}\not\subset B\;(r\in\lbrack s-1\rbrack)}(-1)^{\#B-1}x^{\#A+\#B}\mathfrak{D}(\bm{c}_{(-A)}-\delta_B)\\
        &+\quad\summ{A\subset\lbrack2s\rbrack^1_{\bm{c}}\\A:\even-\odd}\summ{B\subset\lbrack2s\rbrack^{>1}_{\bm{c}}\\\#A+\#B\ge2\\\{2r,2r+1\}\not\subset B\;(r\in\lbrack s-1\rbrack)}(-1)^{\#B-1}x^{\#A+\#B-1}y\mathfrak{D}(\bm{c}_{(-A)}-\delta_B)\\
        &\quad+\summ{A\subset\lbrack2s\rbrack^1_{\bm{c}}\\A:\odd-\even}\summ{B\subset\lbrack2s\rbrack^{>1}_{\bm{c}}\\\#A+\#B\ge2\\\{2r-1,2r\}\not\subset B\;(r\in\lbrack s\rbrack)}(-1)^{\#B}x^{\#A+\#B-1}y\mathfrak{D}(\bm{c}_{(-A)}-\delta_B),
    \end{align}
    and $\mathcal{D}(1)=\mathfrak{D}(\varnothing)=1$. 
\end{dfn}
We describe the notation used here, which follows that of the original definition in \cite{HMSW}. 
\begin{enumerate}[(i)]
    \item For a positive integer $n$ we write  $\lbrack n\rbrack=\{1,2,\dots,n\}$ and set $\lbrack0\rbrack=\varnothing$.
    \item Given a tuple of positive integers $\bm{c}=(c_1,c_2,\dots,c_{2s})$ we set $$\lbrack 2s\rbrack^1_{\bm{c}}=\{i\in\lbrack 2s\rbrack\mid c_i=1\} \quad  \text{and}\quad  \lbrack 2s\rbrack^{>1}_{\bm{c}}=\{i\in\lbrack 2s\rbrack\mid c_i>1\}.$$
    \item We say that a subset $A\subset\lbrack 2s\rbrack$ is \emph{even-odd} (resp. \emph{odd-even}) if whenever an even (resp. odd) integer $i$ belongs to $A$, then $i+1$ also belongs to $A$, and whenever an odd (resp. even) integer $i$ belongs to $A$, then $i-1$ also belongs to $A$.
    \item For a tuple of positive integers $\bm{c}=(c_1,\dots,c_{2s})$ and a subset $A\subset\lbrack 2s\rbrack$, we define $\bm{c}_{(-A)}$ by $(c_{i_1},\dots,c_{i_t})$, where $i_1,\dots,i_t$ are the elements of $\lbrack2s\rbrack\setminus A$ ordered in increasing order.
    \item Furthermore, for a subset $B\subset\lbrack2s\rbrack$, we define $\bm{c}_{(-A)}-\delta_B$ by $(c_{i_1}-\delta_{i_1\in B},\dots,c_{i_t}-\delta_{i_t\in B})$, where $\delta_{i\in B}=1$ if $i\in B$ and $\delta_{i\in B}=0$ if $i\notin B$.
\end{enumerate}

This operator arises from a calculation of a sequence of rational numbers $(\zeta^\diamondsuit_N(\bm{k}))_{N>0}$, called \emph{multiple zeta diamond values}, defined for an admissible index $\bm{k}=(k_1,\dots,k_r)$ by
\begin{align}\label{eq:def dia}
    \zeta^\diamondsuit_N(\bm{k})\coloneqq\sum_{A\subset\lbrack r\rbrack^1_{\bm{k}}}\sum_{(n_1,\dots,n_r)\in S_{r,N}(A)}\left(\prod_{i\in A}\frac{1}{N-n_i}\right)\left(\prod_{i\in\lbrack r\rbrack\setminus A}\frac{1}{n_i^{k_i}}\right),
\end{align}
where 
\begin{align}
    S_{r,N}(A)\coloneqq\left\{(n_1,\dots,n_r)\in\lbrack N-1\rbrack^r\relmid\ \begin{array}{cc} 
n_{i-1}\geq n_{i}  & \text{ if } i\in A, \\ 
n_{i-1} > n_i & \text{ if } i\in \{2,\dots,r\}\setminus A
\end{array} \right\}.
\end{align}
This sequence converges to a multiple zeta value, i.e. $\lim_{N\to\infty}\zeta^\diamondsuit_N(\bm{k})=\zeta(\bm{k})$. Let $\zeta^\diamondsuit:\ha^0\to\QQ^\NN$ be the evaluation map that sends $z_{k_1}\cdots z_{k_r}$ to $(\zeta^\diamondsuit_N(\bm{k}))_{N>0}$. We review some of the results obtained for these objects.
\begin{thm}[\cite{HMSW}]\label{thm:Drop1}
    For any $w\in\ha^0$, it holds $(\mathcal{D}-\id)(w)\in\ker\zeta^\diamondsuit$. In particular, $\zeta(\mathcal{D}(w))=\zeta(w)$ and thus every (admissible) multiple zeta value can be written as a $\ZZ$-linear combination of $\zeta(\bm{k})$ with $\bm{k}\in\mathbb{I}^{\ge2}$.
\end{thm}
Although no algebraic interpretation of the Drop1 operator $\mathcal{D}$ is known so far, the following result is shown in \cite{HMSW}.
\begin{lem}[Corollary of {\cite[Proposition 3.1]{HMSW}}]\label{lem:Dhom}
    For $u\in\ha^0$ and $v\in\ha^{\ge2}$, we have
    \begin{align}
        \mathcal{D}(u\ast v)=\mathcal{D}(u)\ast v.
    \end{align}
\end{lem}

\begin{proof}
    By \cite[Proposition 3.1]{HMSW}, for any $u\in\ha^0$ and $v\in\ha^{\ge2}$, we have
    \begin{align}
        \mathcal{D}(u\ast v)-\mathcal{D}(u)\ast v\in\ker\zeta^\diamondsuit\cap\ha^{\ge2}.
    \end{align}
    Since it is shown that $\{\zeta^\diamondsuit(k_1,\dots,k_r)\mid r\ge1,k_1,\dots,k_r\ge2\}$ are linearly independent (see \cite[Chapter 2]{HMSW}), we have $\ker\zeta^\diamondsuit\cap\ha^{\ge2}=\{0\}$.
\end{proof}

\begin{rem}\label{rem:drop1 identity}
    The Drop1 operator acts as the identity on $\ha^{\ge2}$, i.e. $\mathcal{D}(w)=w$ for any $w\in\ha^{\ge2}$. Indeed, by \Cref{thm:Drop1} we have $(\mathcal{D}-\id)(w)\in\ker\zeta^\diamondsuit\cap\ha^{\ge2}$, and this space is $\{0\}$ as seen in the proof of \Cref{lem:Dhom}. We will use this fact frequently in the following.
\end{rem}

Giving a non-recursive expression for the Drop1 operator seems to be difficult, however, simple expressions exist in some cases.

\begin{lem}[\cite{Se}]\label{lem:Seki}
    For any $2\le j\le i\le r$, we have
    \begin{align}
        \mathcal{D}(z_2^{j-1}z_1z_2^{i-j}z_3 z_2^{r-i})=z_2^{i-j}z_3z_2^{j-2}z_3z_2^{r-i}+z_2^{i-j}z_3z_2^{r-i}z_3z_2^{j-2}+z_2^{r+1}.
    \end{align}
\end{lem}
\begin{proof}
This follows as a special case of \cite[Theorem 1.3]{Se} by setting $w_1=z_2^{r-i} z_3 z_2^{i-j}$ and $w_2= z_2^{j-2}z_3$. We note that the order in \cite{Se} is reversed compared to the convention used here. The result is then obtained by evaluating $w_1 \star w_2$, which yields the right-hand side of the formula above.
\end{proof}


\begin{lem}\label{lem:drop1shufflez2}
For $r\geq 0$ we have
    \begin{align}
    \mathcal{D}(z_2^r\shuffle z_2)=4z_2^r\ast z_2-(r+1)(2r+3)z_2^{r+1}.
\end{align}
\end{lem}
\begin{proof}
By direct calculation one has
\begin{align}
    z_2^r\shuffle z_2=(r+1)z_2^{r+1}+4\sum_{1\le i\le j\le r}z_2^{i-1}z_3z_2^{j-i}z_1z_2^{r-j}.
\end{align}
Thus we get
\begin{align}\label{eq:Jinbo}
    z_2^r\shuffle z_2-4z_2^r\ast z_2=-3(r+1)z_2^{r+1}+4\sum_{1\le i\le j\le r}z_2^{i-1}z_3z_2^{j-i}z_1z_2^{r-j}-4\sum_{1\le i\le r}z_2^{i-1}z_4z_2^{r-i}.
\end{align}
Now consider the Hoffman relation for $4\sum_{1\le i\le r}z_2^{i-1}z_3z_2^{r-i}$, i.e.
\begin{align}\label{eq:Jinbohoffman}
    4\sum_{1\le i\le r}\Ds(z_2^{i-1}z_3z_2^{r-i},z_1)&=-4\sum_{2\le j\le i\le r}z_2^{j-1}z_1z_2^{i-j}z_3z_2^{r-i}-4\sum_{1\le i\le j\le r}z_2^{i-1}z_3z_2^{j-i}z_1z_2^{r-j}\\
    &\quad+4\sum_{1\le i\le r}z_2^{i-1}z_4z_2^{r-i}+8\sum_{1\le j\le i\le r-1}z_2^{j-1}z_3z_2^{i-j}z_3z_2^{r-1-i}-4rz_2^{r+1}.
\end{align}
By using this, we can write \eqref{eq:Jinbo} as follows:
\begin{align}\label{eq:Jinbo2}
    z_2^r\shuffle z_2-4z_2^r\ast z_2&=-4\sum_{1\le i\le r} \Ds(z_2^{i-1}z_3z_2^{r-i},z_1)-4\sum_{2\le j\le i\le r}z_2^{j-1}z_1z_2^{i-j}z_3z_2^{r-i}\\&\quad+8\sum_{1\le j\le i\le r-1}z_2^{j-1}z_3z_2^{i-j}z_3z_2^{r-1-i}-(7r+3)z_2^{r+1}.
\end{align}
In \cite{HMSW}, it is shown that $\ker \mathcal{D}$ contains the linear part of Kawashima’s relations (\cite{Kaw}), which are known to contain the Hoffman relation. Consequently, we have $\Ds(w,z_1) \in \ker \mathcal{D}$ for any $w \in \ha^0$. We now apply the Drop1 operator $\mathcal{D}$ to \eqref{eq:Jinbo2}. The first sum on the right-hand side lies in $\ker\mathcal{D}$, and the words in the third sum are fixed by $\mathcal{D}$ (\Cref{rem:drop1 identity}). For the second sum, \Cref{lem:Seki} gives
\begin{align}
    \mathcal{D}\Bigg(\sum_{2\le j\le i\le r}z_2^{j-1}z_1z_2^{i-j}z_3z_2^{r-i}\Bigg)=\sum_{2\le j\le i\le r}\left(z_2^{i-j}z_3z_2^{j-2}z_3z_2^{r-i}+z_2^{i-j}z_3z_2^{r-i}z_3z_2^{j-2}\right)+\binom{r}{2}z_2^{r+1}.
\end{align}
As $(j,i)$ runs over $2\le j\le i\le r$, the triple of exponents $(i-j,j-2,r-i)$ (resp. $(i-j,r-i,j-2)$) runs exactly once over all $(a,b,c)\in\ZZ_{\geq0}^3$ with $a+b+c=r-2$, and the same holds for the triple $(j-1,i-j,r-1-i)$ as $(j,i)$ runs over $1\le j\le i\le r-1$. Therefore
\begin{align}
    \sum_{2\le j\le i\le r}\left(z_2^{i-j}z_3z_2^{j-2}z_3z_2^{r-i}+z_2^{i-j}z_3z_2^{r-i}z_3z_2^{j-2}\right)=2\sum_{1\le j\le i\le r-1}z_2^{j-1}z_3z_2^{i-j}z_3z_2^{r-1-i},
\end{align}
and applying $\mathcal{D}$ to \eqref{eq:Jinbo2} yields
\begin{align}
    \mathcal{D}(z_2^r\shuffle z_2)-4z_2^r\ast z_2=\left(-4\binom{r}{2}-(7r+3)\right)z_2^{r+1}=-(r+1)(2r+3)z_2^{r+1}.
\end{align}
\vspace{-5pt}
\end{proof}

In the case for a word in $\ha^{\ge2,\mathrm{alm}}$, an explicit formula for $\mathcal{D}$ is given as follows. Note that the previous lemmas (\Cref{lem:Seki}, \Cref{lem:drop1shufflez2}) do not naively follow from this formula.

\begin{thm}\label{thm:drop only one 1}
    For words $z_{\bm{k}}=z_{k_1}\cdots z_{k_r},z_{\bm{l}}=z_{l_1}\cdots z_{l_s}\in\ha^{\geq2}\;(r\geq1,s\geq0)$, we have
    \begin{align}
        \mathcal{D}(z_{\bm{k}}z_1z_{\bm{l}})=z_{\bm{k}}z_1z_{\bm{l}}+\sum_{j=0}^r(-1)^{r-j}z_{k_1}\cdots z_{k_j}\ast\big(((z_{k_r}\cdots z_{k_{j+1}})^\star\ast z_{\bm{l}})\shuffle z_1\big)
    \end{align}
    where the star operator $(\cdot)^\star:\hha^1\rightarrow\hha^1$ is defined by
    \begin{align}
        \ai{k_1,\dots,k_r}{d_1,\dots,d_r}^\star\coloneqq\sum_{\circ_1,\dots,\circ_{r-1} = ,\text{ or }+}\ai{k_1\circ_1\dots\circ_{r-1}k_r}{d_1\circ_1\dots\circ_{r-1} d_r}.
    \end{align}
\end{thm}
\begin{proof}
    For $N>0$ and an index $\bm{k}=(k_1,\dots,k_r)$, define the multiple harmonic sum $\zeta_N(\bm{k})$ by
    \begin{align}
        \zeta_N(\bm{k})\coloneqq\sum_{N>n_1>\cdots>n_r>0}\frac{1}{n_1^{k_1}\cdots n_r^{k_r}}.
    \end{align}
    Note that the multiple harmonic sums satisfy the harmonic product $\ast$. By definition of $\zeta^\diamondsuit_N(\bm{k})$ \eqref{eq:def dia}, for $\bm{k}=(k_1,\dots,k_r),\bm{l}=(l_1,\dots,l_s)\in\mathbb{I}^{\geq2}$, we have
    \begin{align}
        &\zeta^\diamondsuit_N(\bm{k},1,\bm{l})\\
        &=\zeta_N(\bm{k},1,\bm{l})+\sum_{N>n_1>\cdots>n_r\geq m>m_1>\cdots>m_s>0}\frac{1}{n_1^{k_1}\cdots n_r^{k_r}(N-m)m_1^{l_1}\cdots m_s^{l_s}}\\
        &=\zeta_N(\bm{k},1,\bm{l})+\sum_{N>m>m_1>\cdots>m_s>0}\sum_{j=0}^r(-1)^{r-j}\summ{N>n_1>\cdots>n_j>0\\m>n_r\geq\cdots\geq n_{j+1}>0}\frac{1}{n_1^{k_1}\cdots n_r^{k_r}(N-m)m_1^{l_1}\cdots m_s^{l_s}}\\
        &=\zeta_N(\bm{k},1,\bm{l})+\sum_{j=0}^{r}(-1)^{r-j}\zeta_N(k_1,\dots,k_j)\sum_{N>m>0}\frac{1}{N-m}\zeta_m((k_r,\dots,k_{j+1})^\star\ast\bm{l}).
    \end{align}
    Since we have $\sum_{N>m>0}\frac{1}{N-m}\zeta_m(k_1,\dots,k_r)=\zeta_N((k_1,\dots,k_r)\shuffle(1))$ (cf. \cite{Se2}), we have
    \begin{align}
        \zeta^\diamondsuit_N(\bm{k},1,\bm{l})=\zeta_N(\bm{k},1,\bm{l})+\sum_{j=0}^r(-1)^{r-j}\zeta_N\left((k_1,\dots,k_j)\ast(((k_r,\dots,k_{j+1})^\star\ast\bm{l})\shuffle(1))\right).
    \end{align}
    In \cite{HMSW}, it is shown that $\{\zeta^H(\bm{k})\coloneqq(\zeta_N(\bm{k}))_{N>0}\mid\bm{k}\in\mathbb{I}^{\geq2}\}$ forms a basis of the space spanned by all admissible $\zeta^\diamondsuit(\bm{k})$, and $\ker\zeta^\diamondsuit=\Span_\QQ\{(\mathcal{D}-\id)(w)\mid w\in\ha^0\}$. Therefore, we have the conclusion.
\end{proof}

This formula can be written more simply by using the formal multiple Eisenstein series, i.e. in the space $\hha^1$ modulo $\I$.

\begin{lem}\label{lem:drop1 mod I}
    For words $z_{\bm{k}}=z_{k_1}\cdots z_{k_r},z_{\bm{l}}=z_{l_1}\cdots z_{l_s}\in\ha^{\geq2}\;(r\geq1,s\geq0)$, we have in $\hha^1$
    \begin{align}
        \mathcal{D}(z_{\bm{k}}z_1z_{\bm{l}})\equiv z_{\bm{k}}z_1z_{\bm{l}}+\ai{\bm{k},\bm{l}}{\{0\}^{r-1},1,\{0\}^{s}}-\ai{\bm{k},\bm{l}}{\{0\}^{r},1,\{0\}^{s-1}} \mod\I,
    \end{align}
    where the last term is understood as $0$ in the case $s=0$.
\end{lem}
\begin{proof}
    As mentioned in \Cref{rem:symmetril}, in \cite{Ho}, it is shown that the algebra $\hha^1_\ast$ has a Hopf algebra structure and its antipode $S:\hha^1\rightarrow\hha^1$ is given by
    \begin{align}
        S\left(\ai{k_1,\dots,k_r}{d_1,\dots,d_r}\right)=(-1)^r\ai{k_r,\dots,k_1}{d_r,\dots,d_1}^\star.
    \end{align}
    By \Cref{thm:drop only one 1} and \Cref{lem:shuffle with z_1}, we have
    \begin{align}
        \mathcal{D}(z_{\bm{k}}z_1z_{\bm{l}})
        &=z_{\bm{k}}z_1z_{\bm{l}}+\sum_{w_1w_2=z_{\bm{k}}}w_1\ast((S(w_2)\ast z_{\bm{l}})\shuffle z_1)\\
        &=z_{\bm{k}}z_1z_{\bm{l}}+\sum_{w_1w_2=z_{\bm{k}}}w_1\ast\sigma((\sigma(S(w_2)\ast z_{\bm{l}}))\ast \sigma(z_1))-w_1\ast\ai{S(w_2)\ast z_{\bm{l}}}{1,0,\dots,0}.
    \end{align}
    Taking modulo $\I$, the first term inside the sum vanishes since we have the antipode relation:
    \begin{align}
        \sum_{w_1w_2=z_{\bm{k}}}w_1\ast S(w_2)=0.
    \end{align}
    By \Cref{lem:cancel harmonic}, we have
    \begin{align}
        \mathcal{D}(z_{\bm{k}}z_1z_{\bm{l}})\equiv z_{\bm{k}}z_1z_{\bm{l}}+\ai{\bm{k},\bm{l}}{\{0\}^{r-1},1,\{0\}^s}-\ai{\bm{k},\bm{l}}{\{0\}^r,1,\{0\}^{s-1}}\mod\I.
    \end{align}
\end{proof}


More generally, we expect that the Drop1 operator in the formal multiple Eisenstein series can be written as follows.

\begin{conj}\label{conj:drop1 generating}
    For $k_1,\dots,k_r\geq2$, we have
    \begin{align}
        &\sum_{n_1,\dots,n_r\geq0}\mathcal{D}(z_{k_1}z_1^{n_1}\cdots z_{k_r}z_1^{n_r})X_1^{n_1}\cdots X_r^{n_r}\\
        &\equiv\summ{n_1,\dots,n_r\geq0\\j_1,\dots,j_r\geq0}\ai{k_1,\{1\}^{n_1},\dots,k_r,\{1\}^{n_r}}{j_1,\{0\}^{n_1},\dots,j_r,\{0\}^{n_r}}X_1^{n_1}\cdots X_r^{n_r}\frac{X_1^{j_1}}{j_1!}\frac{(X_2-X_1)^{j_2}}{j_2!}\cdots\frac{(X_r-X_{r-1})^{j_r}}{j_r!}\mod\I.
    \end{align}
\end{conj}

\subsection{Explicit formula for the derivative}\label{subsec:derivative}
 
We define the $\QQ$-linear map $\der:\ha^{0}\to\ha^{\ge2}$ by $\der=-\mathcal{D}\circ \dsht$. Note that the restriction of $\der$ onto $\ha^{\ge2}$ is not a derivation on $\ha^{\ge2}_\ast$, and the failure of $\der$ being a derivation can be computed as follows
\begin{align}
    \der(u\ast v)-\der(u)\ast v-u\ast\der(v)=-R(u,v)\qquad(u,v\in\ha^{\ge2}).
\end{align}

The following theorem implies that the map $\der$ provides a closed formula for the derivation $D$ on $\mes^f$ and the derivation $2\pi i\frac{d}{d\tau}$ on $\mes$.

\newtheorem*{mainD}{\Cref{thm:der}}
\begin{mainD}
\begin{enumerate}
    \item For $w\in\ha^{\geq2}$, we have
        \begin{align}
            D(w)\equiv\der(w)\mod{\I}.
        \end{align}
    \item For $w\in\ha^{\geq2}$, we have
        \begin{align}
            2\pi i\frac{d}{d\tau}G(w)=G(\der(w)).
        \end{align}
\end{enumerate}
\end{mainD}

\begin{proof}
\begin{enumerate}
    \item
    Write $z_{\bm{k}}=z_{k_1}\cdots z_{k_r}$. We have
    \begin{align}
        z_{\bm{k}}\shuffle z_2\equiv2\sum_{1\leq j<i\leq r+1}k_j\,z_{k_1}\cdots z_{k_j+1}\cdots z_{k_{i-1}}z_1z_{k_i}\cdots z_{k_r}\mod{\Span_\QQ\{z_{\bm{l}}\mid\bm{l}\in\mathbb{I}^{\ge2}\}}
    \end{align}
    (see the proof of \Cref{lem:R in ha2}). Since $\mathcal{D}$ acts as the identity on $\ha^{\ge2}$ (\Cref{rem:drop1 identity}), by \Cref{lem:drop1 mod I}, $\mathcal{D}(z_{\bm{k}}\shuffle z_2)-z_{\bm{k}}\shuffle z_2$ can be computed as follows
    \begin{align}
        &2\sum_{1\leq j<i\leq r+1}k_j(\mathcal{D}-\id)(z_{k_1}\cdots z_{k_j+1}\cdots z_{k_{i-1}}z_1z_{k_i}\cdots z_{k_r})\\
        &\equiv2\sum_{1\leq j<i\leq r+1}k_j\left(\ai{k_1,\dots,k_j+1,\dots,k_r}{\{0\}^{i-2},1,\{0\}^{r-i+1}}-\ai{k_1,\dots,k_j+1,\dots,k_r}{\{0\}^{i-1},1,\{0\}^{r-i}}\right)\mod{\I}\\
        &=2\sum_{1\leq j\leq r}k_j\ai{k_1,\dots,k_j+1,\dots,k_r}{\{0\}^{j-1},1,\{0\}^{r-j}}\\
        &=2D(z_{\bm{k}}).
    \end{align}
    Since $D(z_{\bm{k}})\equiv\dsht(z_{\bm{k}})\mod{\I}$ (\Cref{thm:sltwo for gf0}), we have
    \begin{align}
        \mathcal{D}(z_{\bm{k}}\shuffle z_2)-z_{\bm{k}}\shuffle z_2\equiv D(z_{\bm{k}})+z_{\bm{k}}\ast z_2-z_{\bm{k}}\shuffle z_2.
    \end{align}
    Therefore we have $\der(z_{\bm{k}})=\mathcal{D}(z_{\bm{k}}\shuffle z_2)-z_{\bm{k}}\ast z_2\equiv D(z_{\bm{k}})\mod{\I}$, where we used $\mathcal{D}(z_{\bm{k}}\ast z_2)=z_{\bm{k}}\ast z_2$ (\Cref{rem:drop1 identity}).
    \item 
    Since $\I\subset\ker\bMES$, by applying $\bMES$ to $D(w)\equiv\der(w)$, we have
    \begin{align}
        \bMES(D(w))=\bMES(\der(w)).
    \end{align}
    By \Cref{prop:der bMES}, it holds $\bMES(D(w))=2\pi i\frac{d}{d\tau}\bMES(w)$ and by \Cref{prop:bMES H2}, we have $\bMES(w)=G(w)$ and $\bMES(\der(w))=G(\der(w))$.
\end{enumerate}
\end{proof}

We give some examples of special cases of the derivative formula.

\begin{ex}
    Let $k\geq2$. In the depth one case, $\der(z_k)$ can be computed as follows. 
    \begin{align}
        \der(z_k) = \mathcal{D}(z_k \shuffle z_2 - z_k \ast z_2)
        &= \mathcal{D}(2k\Ds(z_{k+1},z_1) -  \Ds(z_{k},z_2))\\
        &= (2k-1) z_{k+2} -\sum_{j=2}^k(k+j-1)z_{k-j+2} z_j -z_2 z_k.
    \end{align}
    This agrees with the formulas for the derivative first proven in \cite{BIM} in the formal setting and \cite{Tu} for the classical Eisenstein series.
    \begin{align}
        D(G^f(k))&=(2k-1)G^f(k+2)-\sum_{j=2}^k(k+j-1)G^f(k-j+2,j)-G^f(2,k),\\
        2\pi i\frac{d}{d\tau}G_k(\tau)&=(2k-1)G_{k+2}(\tau)-\sum_{j=2}^k(k+j-1)G_{k-j+2,j}(\tau)-G_{2,k}(\tau).\label{eq:turanderiv}
    \end{align}
\end{ex}


Notice that for any $k\geq 2$ one obtains as a direct consequence of the harmonic product that for any $r\geq 0$ we have $G_{\{k\}^r} \in \QQ[G_{a k} \mid a\geq 1]$, where we write $\{k\}^r = \underbrace{k,\dots,k}_r$. This can be made explicit via generating series, as it is a well-known result for quasi-shuffle algebras (\cite{HI}) that we can write
\begin{align}\label{eq:expG}
1+ \sum_{n=1}^{\infty} G_{\{k\}^n}(\tau) T^n = \exp\left( \sum_{n=1}^{\infty} (-1)^{n-1} G_{nk}(\tau) \frac{T^n}{n} \right).
\end{align}
Applying \eqref{eq:turanderiv} to this would give an explicit, but probably complicated, formula for $2\pi i\frac{d}{d\tau}G_{\{k\}^r}$. It turns out that in the case $k=2$ this can be simplified in the following nice way. 

\begin{ex}
    For any $r\geq1$, we have
    \begin{align}
        \der(z_2^r)=3z_2^r\ast z_2-(r+1)(2r+3)z_2^{r+1}.
    \end{align}
    This follows directly from \Cref{lem:drop1shufflez2}. This agrees with the formula for the derivative first proven in \cite{BY}.
    \begin{align}
        2\pi i\frac{d}{d\tau}G_{\{2\}^r}(\tau)=3G_{\{2\}^r}(\tau)G_2(\tau)-(r+1)(2r+3)G_{\{2\}^{r+1}}(\tau).
    \end{align}
\end{ex}

\subsection{Derivatives of Okounkov's $q$-analogues}\label{subsec:okounkov}
In \cite{O} Okounkov introduces a family of $q$-analogues of multiple zeta values and denotes the $\QQ$-vector space spanned by these by $\qmzv$. As it was shown in \cite{BK}, this space can also be defined by using the $q$-series \begin{align}
    \tilde{g}_{k_1,\dots,k_r}(\tau) &\coloneqq (-2\pi i)^{-(k_1+\cdots+k_r)}g_{k_1,\dots,k_r}(\tau) \\
    &= \frac{1}{(k_1-1)!\cdots(k_r-1)!}\summ{c_1>\cdots>c_r>0\\d_1,\dots,d_r>0}d_1^{k_1-1}\cdots d_r^{k_r-1}q^{c_1 d_1+\cdots+c_r d_r}
\end{align}
as we have
\begin{align}
    \qmzv = \QQ + \Span_\QQ\{\tilde{g}_{k_1,\dots,k_r}(\tau)\mid r\ge1,k_1,\dots,k_r\ge2\}.
\end{align}
In \cite[Conjecture 1]{O} Okounkov gives a conjecture for the dimension of the associated graded algebra with respect to the weight filtration and predicts that $\qmzv$ is closed under $q \frac{d}{dq}$. Using our results from above, we are able to prove the latter claim (\Cref{cor:okounkov}).

The key ingredient is another realization of the formal multiple Eisenstein series, the \emph{combinatorial bi-multiple Eisenstein series} $\cmes\binom{k_1,\dots,k_r}{d_1,\dots,d_r}\in\QQ\llbracket q\rrbracket$ ($k_1,\dots,k_r\geq1$, $d_1,\dots,d_r\geq0$), constructed by the first author and Burmester in \cite[Definition 6.4]{BB}. By \cite[Theorem 6.5]{BB}, their generating series is a symmetril and swap invariant bimould, i.e. the assignment $G^f\binom{k_1,\dots,k_r}{d_1,\dots,d_r}\mapsto\cmes\binom{k_1,\dots,k_r}{d_1,\dots,d_r}$ defines an algebra homomorphism
\begin{align}
    \cmes:\bmes^f\longrightarrow\QQ\llbracket q\rrbracket.
\end{align}
Moreover, by \cite[Proposition 6.29]{BB}, this realization intertwines the derivation $D$ with the operator $q\frac{d}{dq}$, i.e. for any $f\in\bmes^f$ we have
\begin{align}\label{eq:qddq and D}
    q\frac{d}{dq}\,\cmes(f)=\cmes(D(f)).
\end{align}
As before, we write $\cmes(k_1,\dots,k_r)\coloneqq\cmes\binom{k_1,\dots,k_r}{0,\dots,0}$ and call these the \emph{combinatorial multiple Eisenstein series}.

\begin{lem}\label{lem:qmzv cmes}
    We have
    \begin{align}
        \qmzv=\Span_\QQ\{\cmes(k_1,\dots,k_r)\mid r\geq0,\,k_1,\dots,k_r\geq2\}.
    \end{align}
\end{lem}
\begin{proof}
    It follows from the construction of the combinatorial multiple Eisenstein series (see \cite[Section 6]{BB}, cf.\ the proof of \cite[Proposition 6.15]{BB}) that for $k_1,\dots,k_r\geq2$ we have
    \begin{align}\label{eq:cmes structure}
        \cmes(k_1,\dots,k_r)=\tilde{g}_{k_1,\dots,k_r}(\tau)+\summ{0\le s<r\\l_1,\dots,l_s\ge2}\lambda^{(k_1,\dots,k_r)}_{l_1,\dots,l_s}\,\tilde{g}_{l_1,\dots,l_s}(\tau)
    \end{align}
    for some $\lambda^{(k_1,\dots,k_r)}_{l_1,\dots,l_s}\in\QQ$, where we set $\tilde{g}_\varnothing(\tau)\coloneqq1$. In other words, in $\cmes(k_1,\dots,k_r)$ only $\tilde{g}_{l_1,\dots,l_s}$ with $l_1,\dots,l_s\geq2$ (and constants) appear, and the only term in depth $r$ is exactly $\tilde{g}_{k_1,\dots,k_r}$. Since $\cmes(\varnothing)=1$, the expression \eqref{eq:cmes structure} gives the inclusion $\supseteq$. Conversely, inverting \eqref{eq:cmes structure} inductively on the depth $r$ shows that every $\tilde{g}_{k_1,\dots,k_r}$ with $k_1,\dots,k_r\geq2$ is a $\QQ$-linear combination of $\cmes(l_1,\dots,l_s)$ with $0\le s\le r$ and $l_1,\dots,l_s\geq2$, giving the inclusion $\subseteq$.
\end{proof}

\newtheorem*{okounkovcor}{\Cref{cor:okounkov}}
\begin{okounkovcor}
    The space $\qmzv$ is closed under the operator $q\frac{d}{dq}$.
\end{okounkovcor}
\begin{proof}
    By \Cref{lem:qmzv cmes}, it suffices to show $q\frac{d}{dq}\,\cmes(w)\in\qmzv$ for any $w\in\ha^{\geq2}$. By \eqref{eq:qddq and D} and \Cref{thm:der}~(i), we obtain
    \begin{align}
        q\frac{d}{dq}\,\cmes(w)=\cmes(D(w))=\cmes(\der(w)).
    \end{align}
    Since $\der(w)\in\ha^{\geq2}$, the claim follows again from \Cref{lem:qmzv cmes}.
\end{proof}

\section{Another formulation for formal multiple Eisenstein series}\label{sec:another formal MES}

\subsection{Candidate of all linear relations}

In this subsection, we present the family of linear relations that is expected to give all linear relations among multiple Eisenstein series, and some numerical evidence. We define the subspace $\mathsf{R},\mathsf{DR}$ of $\ha^{\ge2}$ by
\begin{align}
    \mathsf{R}&\coloneqq\Span_\QQ\{R(u,v)\mid u,v\in\ha^{\ge2}\},\\
    \mathsf{DR}&\coloneqq\Span_\QQ\{\der^n(R(u,v))\mid n\ge0,u,v\in\ha^{\ge2}\},
\end{align}
and denote by $\mathsf{R}_\ast,\mathsf{DR}_\ast$ the ideals in $\ha^{\ge2}_\ast$ generated by $\mathsf{R}$ and $\mathsf{DR}$, respectively. Note that $\mathsf{R}_\ast,\mathsf{DR}_\ast\subset\ker{G}$ by \Cref{cor:R} and \Cref{thm:der}.
The following table shows the number of linearly independent relations included in $\mathsf{R}_\ast$ and $\mathsf{DR}_\ast$, obtained by the numerical computation. Both families conjecturally yield all relations up to weight $16$; however, $\mathsf{R}_\ast$ fails at weight $17$. In contrast, the number of relations in $\mathsf{DR}_\ast$ coincides with the conjectured number of relations.
\begin{align}
\begin{array}{c|ccccccccccccccccc}
\text{weight }k & 6 & 7 & 8 & 9 & 10 & 11 & 12 & 13 & 14 & 15 & 16 & 17 & 18 & 19 & 20 &
\\\hline
\text{$\#\text{ of conj. relations}=$} & 1 & 1 & 4 & 6 & 13 & 23 & 42 & 74 & 129 & 224 & 382 & 651 & 1098 & 1847 & 3089 &
\\
\text{$\#\text{ of relations in }\mathsf{R}_\ast=$} & 1 & 1 & 4 & 6 & 13 & 23 & 42 & 74 & 129 & 224 & 382 & 650 & 1098 & 1845 & 3086 &
\\
\text{$\#\text{ of relations in }{\mathsf{DR}_\ast}=$} & 1 & 1 & 4 & 6 & 13 & 23 & 42 & 74 & 129 & 224 & 382 & 651 & 1098 & 1847 & 3089 &
\end{array}
\end{align}

\subsection{Another formal multiple Eisenstein space $\widetilde{\mes^f}$}
We introduce an alternative formal multiple Eisenstein space, different from $\mes^f$. We define the $\QQ$-algebra $\widetilde{\mes^f}$ by
\begin{align}
\widetilde{\mes^f}\coloneqq\quotient{\ha^{\ge2}_\ast}{\mathsf{DR}_\ast}.
\end{align}
Based on the results from the previous sections, we have surjective algebra homomorphisms $\widetilde{\mes^f}\twoheadrightarrow\mes^f\twoheadrightarrow\mes$. Indeed, for the first map one has to show that every element of $\mathsf{DR}_\ast$ vanishes in $\mes^f$, i.e. lies in $\I$. By \Cref{thm:sltwo for gf0} we have $\dsht(w)\equiv D(w)\mod\I$ for any $w\in\ha^1$, and since $D$ is a derivation on $\hha^1_\ast$, we get $R(u,v)\in\I$ for any $u,v\in\ha^{\ge2}$. Moreover, the $\sigma$-equivariant derivation $D$ satisfies $D(\I)\subset\I$, so that \Cref{thm:der} inductively yields $\der^n(R(u,v))\in\I$ for all $n\ge0$. The second map is induced by the realization $\bMES$ (\Cref{prop:bMES H2}).


By \Cref{thm:der}, one can confirm the commutator relation $\lbrack\delta,\der\rbrack=W$ in $\mes^f$. Although the space $\mes^f$ is a quotient of $\ha^{\geq2}$, the commutator relation actually holds in $\ha^{\geq2}$.
\begin{lem}\label{lem:commutator2}
    For $w\in\ha^{\ge2}$, we have
    \begin{align}
        \lbrack\delta,\der\rbrack(w)=W(w).
    \end{align}
\end{lem}
\begin{proof}
    Let $w=x^{k_1-1}y\cdots x^{k_r-1}y=x^{k_1-1}yu$. Then, using \Cref{rem:drop1 identity}, we get
    \begin{align}
        \lbrack\delta,\der\rbrack(w)&=-\delta\circ\mathcal{D}\circ \dsht(w)+\mathcal{D}\circ \dsht\circ\delta(w)\\
        &=-\delta(w\ast z_2)+\delta\circ\mathcal{D}(w\shuffle z_2)-\frac{1}{2}\mathbf{1}_{k_1=2}u\ast z_2+\frac{1}{2}\mathbf{1}_{k_1=2}\mathcal{D}(u\shuffle z_2)\\
        &=\frac{1}{2}w+\delta\circ\mathcal{D}(w\shuffle z_2)+\frac{1}{2}\mathbf{1}_{k_1=2}\mathcal{D}(u\shuffle z_2).
    \end{align}
    Thus, it suffices to prove
    \begin{align}
        \delta\circ\mathcal{D}(w\shuffle z_2)=W(w)-\frac{1}{2}w-\frac{1}{2}\mathbf{1}_{k_1=2}\mathcal{D}(u\shuffle z_2).
    \end{align}
    By the explicit formula for $w\shuffle z_2$, this equation is equivalent to the following.
    \begin{align}\label{eq:deltaD}
        &\sum_{1\le i\le j\le r}k_i\delta\circ\mathfrak{D}(k_1-1,1,\dots,k_i,1,\dots,k_j-1,2,\dots,k_r-1,1)\\&=-\frac{1}{2}\mathbf{1}_{k_1=2}\sum_{2\le i\le j\le r}k_i\mathfrak{D}(k_2-1,1,\dots,k_i,1,\dots,k_j-1,2,\dots,k_r-1,1)+\frac{1}{2}W(w).
    \end{align}
    We prove this. Writing $\bm{c}_{i,j}=(k_1-1,1,\dots,k_i,1,\dots,k_j-1,2,\dots,k_r-1,1)$ we get
    \begin{align}
        \delta\circ\mathfrak{D}(\bm{c}_{i,j})&=\frac{1}{2}\Bigg(\summ{A\subset\lbrack2r\rbrack^1_{\bm{c}_{i,j}}\\A=\{2l,2l+1\}\;(l\in\lbrack r-1\rbrack)}-\summ{A\subset\lbrack2r\rbrack^1_{\bm{c}_{i,j}}\\A=\{2l-1,2l\}\;(l\in\lbrack r\rbrack)}\Bigg)\mathfrak{D}(\bm{c}_{i,j(-A)})\\
        &\quad+\frac{1}{2}\Bigg(\summ{B\subset\lbrack2r\rbrack^{>1}_{\bm{c}_{i,j}}\\\#B=2,B\ne\{2l,2l+1\}}-\summ{B\subset\lbrack2r\rbrack^{>1}_{\bm{c}_{i,j}}\\\#B=2,B\ne\{2l-1,2l\}}\Bigg)\mathfrak{D}(\bm{c}_{i,j}-\delta_B).
    \end{align}
    It remains to show
    \begin{align}
        &\sum_{1\le i\le j\le r}k_i    \Bigg(\summ{A\subset\lbrack2r\rbrack^1_{\bm{c}_{i,j}}\\A=\{2l,2l+1\}\;(l\in\lbrack r-1\rbrack)}-\summ{A\subset\lbrack2r\rbrack^1_{\bm{c}_{i,j}}\\A=\{2l-1,2l\}\;(l\in\{2,\dots,r\})}\Bigg)\mathfrak{D}(\bm{c}_{i,j(-A)})=0,\label{eq:vanish1}\\
        &\sum_{1\le i<j\le r}k_i\Bigg(\summ{B\subset\lbrack2r\rbrack^{>1}_{\bm{c}_{i,j}}\\\#B=2,B\ne\{2l,2l+1\}}-\summ{B\subset\lbrack2r\rbrack^{>1}_{\bm{c}_{i,j}}\\\#B=2,B\ne\{2l-1,2l\}}\Bigg)\mathfrak{D}(\bm{c}_{i,j}-\delta_B)=0\label{eq:vanish2},
    \end{align}
    since assuming the above identities, we have
    \begin{align}
        \sum_{1\le i\le j\le r}k_i&\delta\circ\mathfrak{D}(\bm{c}_{i,j})=-\frac{1}{2}\sum_{1\le i\le j\le r}k_i\summ{A\subset\lbrack2r\rbrack^1_{\bm{c}_{i,j}}\\A=\{1,2\}}\mathfrak{D}(\bm{c}_{i,j(-A)})+\frac{1}{2}\sum_{1\le i\le r}k_i\mathfrak{D}(\bm{c}_{i,i}-\delta_{\{2i-1,2i\}})\\
        &=-\frac{1}{2}\mathbf{1}_{k_1=2}\sum_{2\le i\le j\le r}k_i\mathfrak{D}(k_2-1,1,\dots,k_i,1,\dots,k_j-1,2,\dots,k_r-1,1)+\frac{1}{2}W(w).
    \end{align}
    We prove \eqref{eq:vanish1} and \eqref{eq:vanish2}. First, the left-hand side of \eqref{eq:vanish1} is equal to
    \begin{align}
        \sum_{i=1}^rk_i\bigg(\sum_{j=i}^r\summ{l\in\lbrack r-1\rbrack\setminus\{i,j+1\}\\k_l=2}\mathfrak{D}(\bm{c}_{i,j(-\{2l,2l+1\})})-\sum_{j=i}^r\summ{l\in\{2,\dots,r\}\setminus\{i,j\}\\k_l=2}\mathfrak{D}(\bm{c}_{i,j(-\{2l-1,2l\})})\bigg).
    \end{align}
    We know that the sum inside vanishes for each $i\in\lbrack r\rbrack$.
    Second, when $i<j$, we have
    \begin{align}
        \Bigg(\summ{B\subset\lbrack2r\rbrack^{>1}_{\bm{c}_{i,j}}\\\#B=2,\\B\ne\{2l,2l+1\}}-\summ{B\subset\lbrack2r\rbrack^{>1}_{\bm{c}_{i,j}}\\\#B=2,\\B\ne\{2l-1,2l\}}\Bigg)\mathfrak{D}(\bm{c}_{i,j}-\delta_B)
        &=\mathbf{1}_{\substack{k_j\ge3\\k_{j+1}=2}}\mathfrak{D}(\bm{c}_{i,j}-\delta_{\{2j-1,2j\}})-\mathbf{1}_{\substack{k_j=2\\k_{j+1}\ge3}}\mathfrak{D}(\bm{c}_{i,j}-\delta_{\{2j,2j+1\}})\\
        &\quad+\mathbf{1}_{k_j,k_{j+1}\ge3}\left(\mathfrak{D}(\bm{c}_{i,j}-\delta_{\{2j-1,2j\}})-\mathfrak{D}(\bm{c}_{i,j}-\delta_{\{2j,2j+1\}})\right)
    \end{align}
    Therefore, the left-hand side of \eqref{eq:vanish2} is equal to
    \begin{align}
        &\sum_{i=1}^{r-1}k_i\bigg(\sum_{j=i+1}^r\mathbf{1}_{\substack{k_j\ge3\\k_{j+1}=2}}\mathfrak{D}(\bm{c}_{i,j}-\delta_{\{2j-1,2j\}})-\mathbf{1}_{\substack{k_j=2\\k_{j+1}\ge3}}\mathfrak{D}(\bm{c}_{i,j}-\delta_{\{2j,2j+1\}})\\
        &\qquad+\mathbf{1}_{k_j,k_{j+1}\ge3}\left(\mathfrak{D}(\bm{c}_{i,j}-\delta_{\{2j-1,2j\}})-\mathfrak{D}(\bm{c}_{i,j}-\delta_{\{2j,2j+1\}})\right)\bigg).
    \end{align}
    We know the sum inside vanishes for each $i\in\lbrack r-1\rbrack$.
\end{proof}

\newtheorem*{mainE}{\Cref{thm:main2}}
\begin{mainE}
    $\widetilde{\mes^f}$ is an $\sltwo$-algebra, whose $\sltwo$-triple is $(W,\der,\delta)$.
\end{mainE}

\begin{proof}
    By \Cref{lem:commutator2}, it suffices to show the three derivations $W,\der,\delta$ are well-defined in $\widetilde{\mes^f}$. It is easy to see $W$ and $\der$ are well-defined. For $\delta$, it remains to show $\delta(\der^i(R(u,v)))\in\mathsf{DR}_\ast$ for any $i\ge0$, $u,v\in\ha^{\ge2}$. Using the commutator relations $\lbrack\delta,\der\rbrack=W$ and $\lbrack W,\der\rbrack=2\der$ repeatedly, we have
    \begin{align}
        \delta\circ \der^i=\der^i\circ\delta+iW\circ \der^{i-1}-2i\der^{i-1}.
    \end{align}
    Therefore, by \Cref{lem:deltaR}, we have
    \begin{align}
        \delta(\der^i(R(u,v)))=\der^i(R(\delta(u),v))+\der^i(R(u,\delta(v)))+iW\der^{i-1}(R(u,v))-2i\der^{i-1}(R(u,v))\in\mathsf{DR}_\ast,
    \end{align}
    and thus, $\delta$ is well-defined on $\widetilde{\mes^f}$.
\end{proof}

We conclude this subsection by summarizing the relationships between the algebras introduced so far and the remaining problems.

\begin{conj}
    $\mes$, $\mes^f$ and $\widetilde{\mes^f}$ are isomorphic to each other as $\sltwo$-algebras.
\end{conj}
If $W^\prime$ and $\delta^\prime$ are both well-defined, then it follows that $\mes$ is an $\sltwo$-algebra. Indeed, the maps $\mes^f\twoheadrightarrow\mes$ and $\widetilde{\mes^f}\twoheadrightarrow\mes$ intertwine the linear maps $(W,\der,\delta)$ and $(W^\prime,D^\prime,\delta^\prime)$ by \Cref{thm:der} for $\der$ and by definition for $W,\delta$. And thus, the maps induce the same $\sltwo$-algebra structure on $\mes$. The well-definedness of $W^\prime$ is equivalent to the direct sum conjecture over $\QQ$, i.e. $\mes=\bigoplus_{k\geq0}\mes_k$. The well-definedness of $\delta^\prime$ is related to the transcendence of the weight $2$ Eisenstein series $G_2$.
\begin{lem}
    The following are equivalent:
    \begin{enumerate}
        \item $\delta^\prime$ is well-defined.
        \item $-\frac{1}{2}\frac{\partial}{\partial G_2}$ is well-defined.
        \item $G_2$ is transcendental over $\ker(\delta^\prime)=\Span_\QQ\{G_{k_1,\dots,k_r}\in\mes\mid k_1\geq3\}$.
    \end{enumerate}
    And if one of them is true, we have $\delta^\prime=-\frac{1}{2}\frac{\partial}{\partial G_2}$.
\end{lem}
\begin{proof}
    Let $\mes^{\mathrm{adm}}$ be the space spanned by $G_{k_1,\dots,k_r}(\tau)$ with $k_1\geq3$, $k_2,\dots,k_r\geq2$. By the harmonic product, we have $\mes=\mes^{\mathrm{adm}}\lbrack G_2\rbrack$. Thus, a derivation $d\in\Der(\mes)$ such that $d(\mes^{\mathrm{adm}})=0$ is uniquely determined by $d(G_2)$ and always in the form $d=d(G_2)\cdot\frac{\partial}{\partial G_2}$. Therefore the equivalence \textit{(i)} $\Leftrightarrow$ \textit{(ii)} follows and $\delta^\prime=-\frac{1}{2}\frac{\partial}{\partial G_2}$ since $\delta^{\prime}(G_2)=-\frac{1}{2}$. The equivalence \textit{(ii)} $\Leftrightarrow$ \textit{(iii)} follows immediately since $\mes$ is a free polynomial ring over $\mes^{\mathrm{adm}}$ generated by $G_2$ if $\delta^\prime$ is well-defined.
\end{proof}

\subsection{Multiple Eisenstein-diamond series}
The connection between the Drop1 operator and multiple Eisenstein series can be elegantly reformulated by introducing the \emph{multiple Eisenstein-diamond series}, inspired by the multiple zeta-diamond values $\zeta^\diamondsuit$ introduced in \cite{HMSW}.

To this end, we define the $\QQ$-linear map $G^\diamondsuit \coloneqq G \circ \mathcal{D}$ by
\begin{align}
    G^\diamondsuit \colon \ha^0 &\longrightarrow \mes \\
       w &\longmapsto G(\mathcal{D}(w)).
\end{align}

\begin{dfn}
    For an integer $r\ge1$ and an index $(k_1,\dots,k_r)\in(\ZZ_{>0})^r$ with $k_1\ge2$, we define the \emph{multiple Eisenstein-diamond series} by
    \begin{align}
        G^\diamondsuit_{k_1,\dots,k_r}(\tau) \coloneqq G^\diamondsuit(z_{k_1}\cdots z_{k_r}). 
    \end{align}
    Note that for indices with $k_1,\dots,k_r \ge 2$, we have $G^\diamondsuit_{k_1,\dots,k_r}(\tau)=G_{k_1,\dots,k_r}(\tau)$.
\end{dfn}

By \eqref{eq:mesfourier} and the fact that multiple zeta values satisfy the Drop1 relations $\zeta(\mathcal{D}(w))=\zeta(w)$, we obtain the following Fourier expansion for $k_1 \ge 2$ and $k_2,\dots,k_r \ge 1$:
\begin{align}
 G^\diamondsuit_{k_1,\dots,k_r}(\tau) = \zeta(k_1,\dots,k_r)+\sum_{n\ge1}a_nq^n\qquad(a_n\in\mz\lbrack2\pi i\rbrack).
\end{align}
A natural question arises as to whether there exists an explicit definition of $G^\diamondsuit$ in terms of lattice sums that does not rely on the operator $\mathcal{D}$.
Based on the results presented in this work and numerical experiments, we propose the following conjecture.
\begin{conj}\label{conj:diamondmes}
    \begin{enumerate}
        \item $G^\diamondsuit:\ha^0_\ast\rightarrow\mes$ is an algebra homomorphism.
            \item  For any $w\in\ha^0$ we have $$2\pi i\frac{d}{d\tau}G^\diamondsuit(w)=G^\diamondsuit(w\shuffle z_2-w\ast z_2).$$
        \item All relations among $G^\diamondsuit$ are a consequence of the Drop1 relations and products, i.e. 
        $$\ker G^\diamondsuit=\mathsf{Drop1}_\ast,$$
    where $\mathsf{Drop1}_\ast\coloneqq\Span_\QQ\{(\mathcal{D}(u)-u)\ast v\mid u,v\in\ha^0\}.$
    \end{enumerate}

\end{conj}
By \Cref{lem:Dhom}, we have $\mathcal{D}(u\ast v)=\mathcal{D}(u)\ast v$ for $u\in\ha^0$ and $v\in\ha^{\ge2}$, which implies $G^\diamondsuit(u \ast v) = G^\diamondsuit(u) G^\diamondsuit(v)$. However, for general $u,v \in \ha^0$, the statement in \Cref{conj:diamondmes} (i) is non-trivial. Consider, for example, the case $u=v=z_2 z_1$. We have
\begin{align}
 G^\diamondsuit(u \ast v) &=  2 G_{6} + 3 G_{4,2} - 4 G_{3,3},\\
 G^\diamondsuit(u)  G^\diamondsuit(v) &= 2 G_{3,3} + G_6.
\end{align}
These two expressions are equal by virtue of the first relation among multiple Eisenstein series in weight $6$, as given in \Cref{ex:wt6and7}. Assuming \Cref{conj:diamondmes} (i), we can uniquely extend $G^\diamondsuit$ to an algebra homomorphism $G^\diamondsuit \colon \ha^1_\ast \to \mes[T]$ by sending $z_1 \mapsto T$. While setting $T=g_1(\tau)$ would be consistent with the choices for $G^\shuffle$ and $G^\ast$, $G^\diamondsuit$ gives a different way of regularization. This difference is evident even for admissible words; for instance, $G^\diamondsuit(z_2 z_1) = G_3$, whereas we have
$
G^\shuffle(z_2 z_1) = G^\ast(z_2 z_1)= G_3 +\pi i \frac{d}{d\tau}G_1.
$

\appendix
\section{Complements}  \label{sec:appendix}
This appendix collects the proofs of auxiliary results and technical computations omitted from the main sections.
\subsection{Omitted proofs}

\begin{lem}\label{lem:R in ha2}
    For any $u,v\in\ha^{\ge2}$, we have $R(u,v)\in\ha^{\ge2}$.
\end{lem}

\begin{proof}
    It suffices to show
    \begin{align}\label{eq:shffletwo}
        (\bm{k}\ast\bm{k}^\prime)\shuffle(2)\equiv(\bm{k}\shuffle(2))\ast\bm{k}^\prime+(\bm{k}^\prime\shuffle(2))\ast\bm{k}\pmod{\Span_\QQ\mathbb{I}^{\ge2}}
    \end{align}
    for $\bm{k}=(k_1,\dots,k_r),\bm{k}^\prime=(k_{r+1},\dots,k_{r+s})\in\mathbb{I}^{\ge2}$, where here and in the following congruences are taken modulo $\Span_\QQ\mathbb{I}^{\ge2}$.
    Since it holds
    \begin{align}
        (k_1,\dots,k_r)\shuffle(2)\equiv\sum_{1\le{}i\le{}j\le{}r}2k_i(k_1,\dots,k_i+1,\dots,k_j,1,k_{j+1},\dots,k_r),
    \end{align}
    we have the following congruence
    \begin{align}
        (\bm{k}\ast\bm{k}^\prime)\shuffle(2)\equiv\sum_{\sigma\in Sh^{\le}(r,s)}\!\sum_{1\le i\le j\le N}\!2k_{\sigma^{-1}(i)}(k_{\sigma^{-1}(1)},\dots,k_{\sigma^{-1}(i)+1},\dots,k_{\sigma^{-1}(j)},1,k_{\sigma^{-1}(j+1)},\dots,k_{\sigma^{-1}(N)}),
    \end{align}
    where $k_{\sigma^{-1}(i)}=\sum_{j\in\sigma^{-1}(i)}k_j$ and
    \begin{align}
    Sh^{\le}(r,s)=\left\{\sigma:\{1,\dots,r+s\}\twoheadrightarrow\{1,\dots,N\}\relmid\;\begin{matrix}1\le N\le r+s\\\sigma(1)<\cdots<\sigma(r)\\\sigma(r+1)<\cdots<\sigma(r+s)\end{matrix}\;\right\}.
    \end{align}
    The depth $N+1$ $(0<N<r+s)$ part that includes 1 in $n$-th component $(2\le n\le N+1)$ can be computed as follows:
    \begin{align}\label{eq:astshuffle2}
        \sum_{\substack{\sigma:\{1,\dots,r+s\}\twoheadrightarrow\{1,\dots,N\}\\\sigma(1)<\cdots<\sigma(r)\\\sigma(r+1)<\cdots<\sigma(r+s)}}\sum_{1\le i<n}2k_{\sigma^{-1}(i)}(k_{\sigma^{-1}(1)},\dots,k_{\sigma^{-1}(i)+1},\dots,k_{\sigma^{-1}(n-1)},1,k_{\sigma^{-1}(n)},\dots,k_{\sigma^{-1}(N)}).
    \end{align}
    On the other hand, the right-hand side of \eqref{eq:shffletwo} is congruent to
    \begin{align}
        \sum_{1\le i\le j\le r}2k_i\sum_{\substack{\sigma\in Sh^\le(r+1,s)\\\sigma^{-1}(^\exists n)=\{j\}}}(k^\prime_{\sigma^{-1}(1)},\dots,k^\prime_{\sigma^{-1}(N)})
        +\sum_{r+1\le i\le j\le r+s}2k_i\sum_{\substack{\sigma\in Sh^\le(r,s+1)\\\sigma^{-1}(^\exists n)=\{j+1\}}}(k^{\prime\prime}_{\sigma^{-1}(1)},\dots,k^{\prime\prime}_{\sigma^{-1}(N)})
    \end{align}
    where
    \begin{align}
        k^\prime_l=\begin{cases}k_i+1&l=i\\1&l=j+1\\k_{l-1}&j+1<l\le r+1\\k_l&\text{otherwise}\end{cases},\quad
        k^{\prime\prime}_l=\begin{cases}k_i+1&l=i\\1&l=j+1\\k_{l-1}&j+1<l\le r+s+1\\k_l&\text{otherwise}\end{cases}.
    \end{align}
    The component of depth $N+1$ whose $n$-th entry equals $1$ is computed as
    \begin{align}
        &\sum_{1\le i\le j\le r}2k_i\sum_{\substack{\sigma:\{1,\dots,r+s+1\}\twoheadrightarrow\{1,\dots,N+1\}\\\sigma(1)<\cdots<\sigma(j+1)=n<\cdots<\sigma(r+1)\\\sigma(r+2)<\cdots<\sigma(r+s+1)}}(k^\prime_{\sigma^{-1}(1)},\dots,k^\prime_{\sigma^{-1}(n-1)},1,k^\prime_{\sigma^{-1}(n+1)},\dots,k^\prime_{\sigma^{-1}(N+1)})\\
        +&\sum_{r+1\le i\le j\le r+s}2k_i\sum_{\substack{\sigma:\{1,\dots,r+s+1\}\twoheadrightarrow\{1,\dots,N+1\}\\\sigma(1)<\cdots<\sigma(r)\\\sigma(r+1)<\cdots<\sigma(j+1)=n<\cdots<\sigma(r+s+1)}}(k^{\prime\prime}_{\sigma^{-1}(1)},\dots,k^{\prime\prime}_{\sigma^{-1}(n-1)},1,k^{\prime\prime}_{\sigma^{-1}(n+1)},\dots,k^{\prime\prime}_{\sigma^{-1}(N+1)})\\
        =&\sum_{\substack{\sigma:\{1,\dots,r+s\}\twoheadrightarrow\{1,\dots,N\}\\\sigma(1)<\cdots<\sigma(r)\\\sigma(r+1)<\cdots<\sigma(r+s)}}\sum_{j=1}^n\sum_{i\in\sigma^{-1}(j)}2k_i(k_{\sigma^{-1}(1)},\dots,k_{\sigma^{-1}(j)+1},\dots,k_{\sigma^{-1}(n-1)},1,k_{\sigma^{-1}(n)},\dots,k_{\sigma^{-1}(N)}).
    \end{align}
    This is exactly \eqref{eq:astshuffle2}.    
\end{proof}

\begin{lem}\label{lem:gilat}
    Let $A,B$ be $v$-moulds, the coefficient of $v_1^{k_1-1}\cdots v_r^{k_r-1}$ in $\gilat_B(A)$ is given as follows.
    \begin{align}
        &\summ{0=t_0<q_1\le t_1<q_2\le t_2<\cdots<q_s\le t_s=r\\(1\le s\le r)}\summ{n_{t_{j-1}+1},\dots,n_{t_j}\ge1\\n_{t_{j-1};t_j}=k_{t_{j-1};t_j}\\(1\le j\le s)}a(n_{q_1},\dots,n_{q_s})\\
        &\times\prod_{j=1}^s\bigg\{(-1)^{k_{t_{j-1};t_j}+k_{q_j}+n_{t_{j-1};q_j-1}}\Bigg(\prod_{\substack{p=t_{j-1}+1\\p\ne q_j}}^{t_j}\binom{n_p-1}{k_p-1}\Bigg)b(n_{t_{j-1}+1},\dots,n_{q_j-1})b(n_{t_j},\dots,n_{q_j+1})\bigg\}.
    \end{align}
\end{lem}
\begin{proof}
    By definition of $\gilat$, we have
    \begin{align}
        \gilat_B(A)(v_1,\dots,v_r)&=\summ{\bm{v}=\bm{a}_1\bm{b}_1\bm{c}_1\cdots \bm{a}_s\bm{b}_s\bm{c}_s\\\bm{b}_i,\bm{c}_i\bm{a}_{i+1}\ne\emptyset}A(\bm{b}_1\cdots\bm{b}_s)\prod_{j=1}^sB(\bm{a}_j\rfloor)B_-(\lfloor\bm{c}_j)\\
        &=\summ{0=t_0<q_1\le t_1<\cdots<q_s\le t_s<r+1\\(1\le s\le r)}A(v_{q_1},\dots,v_{q_s})\\
        &\qquad\times\prod_{i=1}^sB(v_{t_{i-1}+1}-v_{q_i},\dots,v_{q_i-1}-v_{q_i})B_-(v_{q_i+1}-v_{q_i},\dots,v_{t_i}-v_{q_i}).\\
    \end{align}
    Each mould can be expanded as follows.
    \begin{align}
        A(v_{q_1},\dots,v_{q_s})&=\sum_{n_{q_1},\dots,n_{q_s}>0}a(n_{q_1},\dots,n_{q_s})v_{q_1}^{n_{q_1}-1}\cdots v_{q_s}^{n_{q_s}-1},\\
        B(v_{t_{j-1}+1}-v_{q_j},\dots,v_{q_j-1}-v_{q_j})&=\sum_{n_{t_{j-1}+1},\dots,n_{q_j-1}>0}b(n_{t_{j-1}+1},\dots,n_{q_j-1})\prod_{p=t_{j-1}+1}^{q_j-1}(v_p-v_{q_j})^{n_p-1},\\
        B_-(v_{q_j+1}-v_{q_j},\dots,v_{t_j}-v_{q_j})&=\sum_{n_{q_j+1},\dots,n_{t_j}>0}(-1)^{n_{q_j;t_j}}b(n_{t_j},\dots,n_{q_j+1})\prod_{p=q_j+1}^{t_j}(v_p-v_{q_j})^{n_p-1}.
    \end{align}
    Expanding total sum using binomial coefficients and extracting the coefficient of $v_1^{k_1-1}\cdots v_r^{k_r-1}$ yields the claim.
\end{proof}

\begin{lem}\label{lem:cancel harmonic}
    Let $\rho_i:\hha^1\rightarrow\hha^1$ $(i\geq0)$ be the $\QQ$-linear map defined by
    \begin{align}
        \rho_i\left(\ai{k_1,\dots,k_r}{d_1,\dots,d_r}\right)\coloneqq\begin{cases}\ai{k_1,\dots,k_i,\dots,k_r}{d_1,\dots,d_i+1,\dots,d_r}&1\leq i\leq r\\0&\otherwise\end{cases}.
    \end{align}
    Then, for $\bm{k},\bm{l}$ with $\dep(\bm{k})=r\geq1,\dep(\bm{l})=s\geq0$, we have
    \begin{align}\label{eq:rho1}
        \sum_{w_1w_2=z_{\bm{k}}}w_1\ast\rho_1(S(w_2)\ast z_{\bm{l}})=\rho_{r+1}(z_{\bm{k}}z_{\bm{l}})-\rho_{r}(z_{\bm{k}}z_{\bm{l}}).
    \end{align}
\end{lem}
\begin{proof}
    We prove by induction of $\dep(\bm{k})=r\geq1$. In the case $r=1$, we have
    \begin{align}
        z_{k}\ast\rho_1(z_{\bm{l}})+\rho_1(S(z_k)\ast z_{\bm{l}})=\ai{k,\bm{l}}{0,1,\{0\}^{s-1}}-\ai{k,\bm{l}}{1,\{0\}^s}.
    \end{align}
    When $r>1$, since the antipode is given by $S(z_{k_1}\cdots z_{k_r})=-\sum_{i=0}^{r-1}S(z_{k_1}\cdots z_{k_i})\ast z_{k_{i+1}}\cdots z_{k_r}$, the left-hand side of \eqref{eq:rho1} is given as follows.
    \begin{align}
        &z_{\bm{k}}\ast\rho_1(z_{\bm{l}})-\sum_{i=0}^{r-1}\sum_{j=i}^{r-1}z_{k_1}\cdots z_{k_i}\ast\rho_1(S(z_{k_{i+1}}\cdots z_{k_j})\ast z_{k_{j+1}}\cdots z_{k_r}\ast z_{\bm{l}})\\
        =&z_{\bm{k}}\ast\rho_1(z_{\bm{l}})-\rho_1(z_{\bm{k}}\ast z_{\bm{l}})-\sum_{j=1}^{r-1}\sum_{w_1w_2=z_{k_1}\cdots z_{k_j}}w_1\ast\rho_1(S(w_2)\ast z_{k_{j+1}}\cdots z_{k_r}\ast z_{\bm{l}}).
    \end{align}
    By induction hypothesis, this equals to
    \begin{align}
        z_{\bm{k}}\ast\rho_1(z_{\bm{l}})-\rho_1(z_{\bm{k}}\ast z_{\bm{l}})-\sum_{j=1}^{r-1}\left(\rho_{j+1}(z_{k_1}\cdots z_{k_j}(z_{k_{j+1}}\cdots z_{k_r}\ast z_{\bm{l}}))-\rho_j(z_{k_1}\cdots z_{k_j}(z_{k_{j+1}}\cdots z_{k_r}\ast z_{\bm{l}}))\right).
    \end{align}
    By considering the harmonic product, the following holds and we have the conclusion.
    \begin{align}
        z_{\bm{k}}\ast\rho_1(z_{\bm{l}})-\sum_{j=1}^{r-1}\rho_{j+1}(z_{k_1}\cdots z_{k_j}(z_{k_{j+1}}\cdots z_{k_r}\ast z_{\bm{l}}))&=\rho_{r+1}(z_{\bm{k}}z_{\bm{l}}),\\
        \rho_1(z_{\bm{k}}\ast z_{\bm{l}})-\sum_{j=1}^{r-1}\rho_j(z_{k_1}\cdots z_{k_j}(z_{k_{j+1}}\cdots z_{k_r}\ast z_{\bm{l}}))&=\rho_r(z_{\bm{k}}z_{\bm{l}}).
    \end{align}
\end{proof}

\subsection{Comparison of two $\ast$-regularization methods}
Given a $\QQ$-algebra $R$ and a family of homomorphisms
\begin{align}
    \{w\mapsto f_w(m)\}_{m>0}
\end{align}
from $\ha^1_\ast$ to $R$, we define for a word $w\in\ha^1$ and $M>0$
\begin{align}
    F_w(M)\coloneqq\summ{w_1\cdots w_s=w\\(1\le s\le\dep(w),w_i\neq\emptyset)}\sum_{M>m_1>\cdots>m_s>0}f_{w_1}(m_1)\cdots f_{w_s}(m_s).
\end{align}
\begin{lem}[\cite{Ba2}]\label{lem:f to F}
    For any $M>0$, the map from $\ha^1_\ast$ to $R$ defined by $w\mapsto F_w(M)$ is an algebra homomorphism.
\end{lem}
Note that a $v$-mould $A$ over $R$ can be identified with a linear map from $\ha^1$ to $R$ that sends $z_{k_1}\cdots z_{k_r}$ to $\langle A\mid(k_1,\dots,k_r)\rangle$, and we use the same notation. For $M>0$, define $v$-moulds $\co{F}(M),F^\ast(M)$ by
\begin{align}
    \co{F}(M)(v_1,\dots,v_r)&\coloneqq(-1)^r\sum_{M>m_1\ge\cdots\ge m_r>0}\mu^{m_1,\dots,m_r}f_{z_1}(m_1)\cdots f_{z_1}(m_r),\\
    F^\ast(M)&\coloneqq\co{F}(M)\times F(M),
\end{align}
where $\mu^{m_1,\dots,m_r}=\frac{1}{r_1!\cdots r_s!}$ for a non-increasing sequence $(m_1,\dots,m_r)=(\{n_1\}^{r_1},\dots,\{n_s\}^{r_s})$.
\begin{lem}\label{lem:coF}
    For any $M>0$, $\co{F}(M):\ha^1_\ast\to R:z_{k_1}\cdots z_{k_r}\mapsto\langle \co{F}(M)\mid(k_1,\dots,k_r)\rangle$ is an algebra homomorphism.
\end{lem}
\begin{proof}
    Consider a family of linear maps $\{w\mapsto\widetilde{f}_w(m)\}_{m>0}$ defined by
    \begin{align}
        \widetilde{f}_w(m)\coloneqq\begin{cases}\frac{(-1)^i}{i!}f_{z_1}(m)^i&w=z_1^i\;(i\ge0),\\0&\otherwise.\end{cases}
    \end{align}
    For $m>0$, $\widetilde{f}(m)$ satisfy the harmonic product since
    \begin{align}
        \widetilde{f}_{z_1^i}(m)\widetilde{f}_{z_1^j}(m)=\frac{(-1)^{i+j}}{i!j!}f_{z_1}(m)^{i+j}=\binom{i+j}{i}\widetilde{f}_{z_1^{i+j}}(m).
    \end{align}
    By \Cref{lem:f to F}, the associated family of linear maps $\{w\mapsto\widetilde{F}_w(M)\}_{M>0}$ is again a family of homomorphisms, and this $\widetilde{F}(M)$ coincides with $\co{F}(M)$.
\end{proof}
\begin{lem}\label{lem:comp}
    For any $M>0$, $F^\ast(M):\ha^1_\ast\to R$ is an algebra homomorphism. Moreover, we have $F^\ast_w(M)=F_{\mathrm{reg}^0_\ast(w)}(M)$ for any $w\in\ha^1$, where the map (actually an algebra homomorphism) $\mathrm{reg}^0_\ast:\ha^1_\ast\to\ha^0_\ast$ is the harmonic regularization map introduced by Ihara--Kaneko--Zagier \cite{IKZ}.
\end{lem}
\begin{proof}
    Note that the multiplication of $v$-moulds $A\times B$ coincides with the convolution product of $A,B\in\Hom_{\QQ\text{-}\mathrm{linear}}(\ha^1,R)$ defined by
    \begin{align}
        A\star_{\mathrm{dec}}B\coloneqq\mu\circ(A\otimes B)\circ\Delta_{\mathrm{dec}},
    \end{align}
    where $\Delta_{\mathrm{dec}}:\ha^1\to\ha^1\otimes\ha^1$ is the deconcatenation coproduct defined by $\Delta_{\mathrm{dec}}(z_{k_1}\cdots z_{k_r})\coloneqq\sum_{i=0}^rz_1\cdots z_{k_i}\otimes z_{k_i+1}\cdots z_{k_r}$. And it is known that $\Delta_\mathrm{dec}$ is an algebra homomorphism in terms of the harmonic product $\ast$. Thus it suffices to show that $\co{F}(M)$ satisfies the harmonic product, and it follows from \Cref{lem:coF}.
    The identity $F^\ast(M)=F(M)\circ\mathrm{reg_\ast^0}$     holds since $F^\ast_{z_1}(M)=0$ and the fact that $\mathrm{reg}_\ast^0:\ha^1_\ast\to\ha^0_\ast$ satisfies the following universality: ``For any $\QQ$-algebra $R$ and any algebra homomorphism $\mathfrak{z}:\ha^0_\ast\to R$, there exists a unique algebra homomorphism $\mathfrak{z}^\ast:\ha^1_\ast\to R$ such that $\mathfrak{z}^\ast(z_1)=0$ and $\mathfrak{z}^\ast=\mathfrak{z}\circ\mathrm{reg}_\ast^0$''. 
\end{proof}

\end{document}